\documentclass[11pt, reqno]{amsart}
\usepackage{amssymb}
\usepackage{amscd}
\usepackage{verbatim,ifthen}
\usepackage{color}
\usepackage{latexsym}
\usepackage{tikz}
\usepackage{tikz-cd}
\usepackage{mathrsfs}
\usepackage{wrapfig}
\usetikzlibrary{shapes}
\usepackage{color}
\usetikzlibrary{arrows.meta}
\usepackage{bbm}
\usetikzlibrary{matrix}
\usetikzlibrary{calc}
\usetikzlibrary{arrows,intersections}
\usepackage{pgfplots}
\usepackage{multicol}
\usepackage{array}
\newcolumntype{M}[1]{>{\centering\arraybackslash}m{#1}}

\usepackage[colorlinks, linkcolor=black, citecolor=magenta, linktocpage,backref=page]{hyperref}

\renewcommand*{\backref}[1]{}
\renewcommand*{\backrefalt}[4]{[{\tiny%
    \ifcase #1 Not cited.%
          \or Cited on page~#2.%
          \else Cited on pages #2.%
    \fi%
    }]}
\usepackage[tikz]{bclogo}
\addtocontents{toc}{\setcounter{tocdepth}{2}}

\usepackage{amsmath}

\numberwithin{equation}{section}

\let\oldtocsection=\tocsection
 
\let\oldtocsubsection=\tocsubsection

\renewcommand{\tocsection}[2]{\hspace{0em}\oldtocsection{#1}{#2}}
\renewcommand{\tocsubsection}[2]{\hspace{1em}\oldtocsubsection{#1}{#2}}

\def\Xint#1{\mathchoice
{\XXint\displaystyle\textstyle{#1}}%
{\XXint\textstyle\scriptstyle{#1}}%
{\XXint\scriptstyle\scriptscriptstyle{#1}}%
{\XXint\scriptscriptstyle\scriptscriptstyle{#1}}%
\!\int}
\def\XXint#1#2#3{{\setbox0=\hbox{$#1{#2#3}{\int}$ }
\vcenter{\hbox{$#2#3$ }}\kern-.6\wd0}}

\def\dashint{\Xint-}
\usepackage{eucal}
\usepackage{calc}  
\usepackage{enumitem} 
\usepackage{tensor}
\usepackage{graphicx,wrapfig,lipsum}
\usepackage{etoolbox}
\usepackage{marginnote}
\usepackage{lipsum}
\makeatletter
\patchcmd{\@mn@margintest}{\@tempswafalse}{\@tempswatrue}{}{}
\patchcmd{\@mn@margintest}{\@tempswafalse}{\@tempswatrue}{}{}
\reversemarginpar 
\makeatother
\usepackage{scrextend}

\makeatletter
\DeclareRobustCommand\widecheck[1]{{\mathpalette\@widecheck{#1}}}
\def\@widecheck#1#2{%
    \setbox\z@\hbox{\m@th$#1#2$}%
    \setbox\tw@\hbox{\m@th$#1%
       \widehat{%
          \vrule\@width\z@\@height\ht\z@
          \vrule\@height\z@\@width\wd\z@}$}%
    \dp\tw@-\ht\z@
    \@tempdima\ht\z@ \advance\@tempdima2\ht\tw@ \divide\@tempdima\thr@@
    \setbox\tw@\hbox{%
       \raise\@tempdima\hbox{\scalebox{1}[-1]{\lower\@tempdima\box
\tw@}}}%
    {\ooalign{\box\tw@ \cr \box\z@}}}
\makeatother

\title{On the Gauduchon Curvature of Hermitian Manifolds}
\author{Kyle Broder}
\address{The University of Queensland,  St. Lucia,  QLD 4067, Australia}
\email{k.broder@uq.edu.au}
\thanks{The first named author was partially supported by an Australian Government Research Training Program (RTP) Scholarship and funding from the Australian Government through the Australian Research Council's Discovery Projects funding scheme (project DP220102530). The second named author was supported an Australian Government Research Training Program (RTP) Scholarship.}
\author{James Stanfield}
\address{The University of Queensland,  St. Lucia,  QLD 4067, Australia}
\email{james.stanfield@uq.net.au}
\usepackage{soul}
\usepackage{letltxmacro}
\makeatletter
\AtBeginDocument{%
  \@ifdefinable{\myorg@nameref}{%
    \LetLtxMacro\myorg@nameref\nameref
    \DeclareRobustCommand*{\nameref}[1]{%
      \emph{\myorg@nameref{#1}}%
    }%
  }%
}
\makeatother

\theoremstyle{plain}
\theoremstyle{definition}
\newtheorem{thm}{{\textbf{Theorem}}}[section]
\newtheorem{thmfix}[thm]{{Theorem}}
\newtheorem{defn}[thm]{{Definition}}

\newtheorem{rmk}[thm]{{Remark}}

\newtheorem{notn}[thm]{{Notation}}
\newtheorem{ex}[thm]{{Example}}
\newtheorem{cor}[thm]{{Corollary}}
\newtheorem{prop}[thm]{{Proposition}}
\newtheorem{lem}[thm]{{Lemma}}
\newtheorem{q}[thm]{{Question}}

\begin{document}

\maketitle

\begin{abstract}
It is shown that many results, previously believed to be properties of the Lichnerowicz Ricci curvature, hold for the Ricci curvature of all Gauduchon connections. We prove the existence of $t$--Gauduchon Ricci-flat metrics on the suspension of a compact Sasaki--Einstein manifold, for all $t \in (-\infty,1)$; in particular, for the Bismut, Minimal, and Hermitian conformal connection. A monotonicity theorem is obtained for the Gauduchon holomorphic sectional curvature, illustrating a maximality property for the Chern connection and furnishing insight into known phenomena concerning hyperbolicity and the existence of rational curves. Moreover, we show a rigidity result for Hermitian metrics which have a pair of Gauduchon holomorphic sectional curvatures that are equal, elucidating a duality implicit in the recent work of Chen--Nie.
\end{abstract}

\thispagestyle{empty}
\tableofcontents

\section{Introduction}

Let $(\mathcal{E},h) \to X$ be a Hermitian (holomorphic) line bundle over a complex manifold $X$. A complex-linear connection $\nabla$ on $\mathcal{E}$ is said to be Hermitian if $\nabla h =0$.  The presence of a holomorphic variation in the fibers of $\mathcal{E}$ is encoded in a complex-linear first-order differential operator $\bar{\partial}^{\mathcal{E}}$ satisfying $\bar{\partial}^{\mathcal{E}} \circ \bar{\partial}^{\mathcal{E}}=0$ \cite{KoszulMalgrange}.  This remarkable fact allows the complex geometry of $(\mathcal{E},h) \to X$ to be studied through the differential geometry of the unique Hermitian connection satisfying $\nabla^{0,1} = \bar{\partial}^{\mathcal{E}}$. When $\mathcal{E} = T^{1,0}X$ is the tangent bundle, this connection is called the \textit{Chern connection}, denoted by ${}^c \nabla$.

There has been a rapidly growing interest in the study of more general Hermitian connections on $T^{1,0}X$.  The most notable example of this is the \textit{Strominger--Bismut} \textit{connection} \cite{Bismut,Strominger1986} -- the unique Hermitian connection ${}^b \nabla$ with totally skew-symmetric torsion -- which has played a role in string theory \cite{IvanovPapadopoulos,Strominger1986}, index theory \cite{Bismut}, and the pluriclosed flow \cite{StreetsVII,StreetsTian, StreetsTianPCF1, StreetsTianPCF2}. 

The \textit{Lichnerowicz connection} \cite{Lichnerowicz} -- the restriction ${}^l \nabla$ of the complexified Levi-Civita connection ${}^{\text{LC}} \nabla \otimes \mathbb{C}$ to $T^{1,0}X$ -- has recently received considerable interest \cite{HeLiuYang,LiuYangGeometry,LiuYangRicci,YangKodairaDimension}. In the aforementioned references, the Ricci curvatures of the Lichnerowicz connection are referred to as the `Levi-Civita Ricci curvatures'. We abandon this terminology, using `Lichnerowicz Ricci curvatures' since the connection was discovered by Lichnerowicz \cite{Lichnerowicz}, and because of the additional confusion, the older terminology introduced. 

Libermann \cite{Libermann1, Libermann2, Libermann3} introduced the \textit{Hermitian conformal connection} ${}^{\text{Hc}} \nabla$ -- The unique Hermitian connection ${}^{\text{Hc}} \nabla$ whose torsion satisfies the Bianchi identity. Among the Gauduchon connections that we will soon discuss,  the Hermitian conformal connection can be characterized by the corresponding Dirac operator being conformally covariant (see \cite{Libermann1, Libermann2,Libermann3, GauduchonHermitianConnections}) or the unique Libermann connection\footnote{We say that a connection is \textit{Libermann} if it satisfies $T_b^{1,1}=0$ and $\mathcal{B}(T_c^{1,1})=a\,d \omega + b\,d^c \omega$ for some $a,b \in \mathbb{R}$.} \cite{Libermann1} whose curvature is invariant under conformal deformation.\\

It was observed by Ehresmann and Libermann, and later expounded upon by Gauduchon (see \cite{GauduchonHermitianConnections} and the references therein) that the infinite-dimensional affine space of Hermitian connections on the tangent bundle of a complex manifold supports a distinguished (real) one-dimensional subspace -- what we now call the \textit{Gauduchon line}.  Let ${}^c \nabla$ and ${}^l \nabla$ respectively denote the Chern and Lichnerowicz connections. The \textit{$t$--Gauduchon connection} ${}^t \nabla$ (or \textit{Gauduchon connection} with \textit{Gauduchon parameter} $t$) is given by \begin{eqnarray*}
{}^t \nabla & : = & t {}^c \nabla + (1-t) {}^l \nabla, \hspace{1cm} t \in \mathbb{R}.
\end{eqnarray*}

In particular, ${}^1 \nabla = {}^c \nabla$ and ${}^0 \nabla = {}^l \nabla$.  This line of connections passes through a number of distinguished Hermitian connections that were discovered independently. For instance,  the Strominger--Bismut connection \cite{Strominger1986,Bismut} is given by ${}^b \nabla = {}^{-1} \nabla$, the Hermitian conformal connection \cite{Libermann1, Libermann2, Libermann3} corresponds to ${}^{\text{Hc}} \nabla =  {}^{\frac{1}{2}} \nabla$, and the \textit{Minimal connection} \cite{GauduchonHermitianConnections} -- the unique Hermitian connection ${}^{\text{min}} \nabla$ whose torsion has smallest pointwise norm -- corresponds to ${}^{\frac{1}{3}} \nabla$.\\ 

There has been gaining interest in studying the curvature of the Gauduchon connections. For instance, in \cite{LafuenteStanfield}, R. Lafuente the second named author completed the classification of compact Hermitian manifolds with flat (or more generally K\"ahler-like) Gauduchon connections. Specifically, the authors showed that except for the cases of a flat Chern or Bismut connection, such manifolds are K\"ahler. In this paper, we focus on other curvature tensors associated with the Gauduchon connections starting with the Ricci curvatures. The curvature of a Gauduchon connection ${}^t\nabla$ on $T^{1,0}X$ is an $\operatorname{End}(T^{1,0}X)$-valued $2$-form ${}^tR \in \Omega^{2}_X\otimes \operatorname{End}(T^{1,0}X)$ on $X$. We define the \emph{first $t$-Gauduchon Ricci form} ${}^t\operatorname{Ric}\in\Omega^{1,1}(X)$ as
\[
{}^t\operatorname{Ric}^{(1)}(u,\overline{v}) := \sqrt{-1}\operatorname{tr}{}^tR(u,\overline{v});\qquad u,v \in T^{1,0}X.
\]
Let us remark that in general this trace also has $(2,0)$ and $(0,2)$ components which we omit in favour of the conventions in \cite{LiuYangRicci,LiuYangLCRF,HeLiuYang,WangYang}. Moreover, in contrast with the Riemannian curvature tensor, the presence of torsion in the Gauduchon connections produces distinct Ricci curvatures by taking various traces (see Definition \ref{def:Ricci}).

Our first main result illustrates that certain cohomological properties which were previously thought to be a property of the Lichnerowicz Ricci curvature (corresponding to Gauduchon parameter $t=0$) hold for all Gauduchon connections. For instance, let us first remind the reader that the \textit{first Aeppli--Chern class} $c_1^{\text{AC}}(X) := c_1^{\text{AC}}(K_X^{-1})$ is a cohomology class in $H_A^{1,1}(X):= \{ \alpha \in \Omega_X^{1,1} : \partial \bar{\partial} \alpha=0\}/\{\partial u + \overline{\partial u} : u \in \mathcal{C}^{\infty}(X) \}$. This was studied by Liu--Yang \cite{LiuYangRicci} via the first Lichnerowicz Ricci curvature ${}^l \text{Ric}_{\omega}^{(1)} : = {}^0 \text{Ric}_{\omega}^{(1)}$. We obtain the following extension of \cite[Theorem 3.14]{LiuYangRicci}:

\begin{thm}\label{MainThm2}
Let $(X, \omega)$ be a Hermitian manifold.  The first Gauduchon Ricci form ${}^t \text{Ric}_{\omega}^{(1)}$ represents $c_1^{\text{AC}}(K_X^{-1}) \in H_A^{1,1}(X)$ for all $t \in \mathbb{R}$.  Moreover, \begin{itemize}
\item[(i)] ${}^t \text{Ric}_{\omega}^{(1)}$ is $d$--closed if and only if $t = 1$, or $\partial \bar{\partial} \bar{\partial}^{\ast}\omega=0$.
\item[(ii)] If $\bar{\partial} \partial^{\ast} \omega = 0$, then ${}^t \text{Ric}_{\omega}^{(1)}$ represents the $c_1(K_X^{-1}) \in H_{\text{DR}}^2(X,\mathbb{R})$, i.e.,  $c_1(K_X^{-1})=c_1^{\text{AC}}(K_X^{-1})$.
\item[(iii)] If $\omega$ is conformally balanced, then ${}^t \text{Ric}_{\omega}^{(1)}$ represents $c_1(K_X^{-1}) \in H_{\bar{\partial}}^{1,1}(X)$ and also the first Bott--Chern class $c_1^{\text{BC}}(K_X^{-1}) \in H_{\text{BC}}^{1,1}(X)$. 
\item[(iv)] ${}^t \text{Ric}_{\omega}^{(1)} = {}^s \text{Ric}_{\omega}^{(1)}$ for $t \neq s$ if and only if $\omega$ is balanced.
\end{itemize} \end{thm}

This refined understanding of the Gauduchon Ricci curvatures permits us to obtain the following obstruction to the existence of a locally conformally K\"ahler structure. We remind the reader that a locally conformally K\"ahler metric is a Hermitian metric that is conformal to a K\"ahler metric on an open neighborhood around any given point. We say a complex manifold is locally conformally K\"ahler if it admits a locally conformally K\"ahler metric. We encourage the reader to consult the recent book by Ornea--Verbitsky \cite{OrneaVerbitsky} for a reference to the subject. 

\begin{thm}
Let $(X,\omega)$ be a compact Hermitian manifold with ${}^t \text{Ric}_{\omega}^{(1)}=0$ for some $t \in \mathbb{R} \backslash \{ 1 \}$. If $c_1(K_X) \neq 0$ in $H_{\text{DR}}^2(X,\mathbb{R})$, then $X$ is not locally conformally K\"ahler.
\end{thm}

\begin{rmk}
One cannot relax the assumption of the above theorem to the vanishing of the first Aeppli Chern class $c_1^{\text{AC}} =0$ in $H_A^{1,1}(X)$ (see Remark \ref{relaxAeppli} for more details). Moreover, if ${}^t\text{Ric}_{\omega}^{(1)}=0$, then $c_1^{\text{BC}}(K_X)=0$, which implies that $c_1(K_X)=0$. 
\end{rmk}

It is natural to seek compact non-K\"ahler examples of Hermitian metrics for which ${}^t\operatorname{Ric} \equiv 0$. In \cite{LiuYangRicci,LiuYangLCRF,HeLiuYang}, Liu and Yang gave an explicit construction of \emph{first Lichnerowicz Ricci-flat} metrics (i.e. satisfying ${}^0\operatorname{Ric}^{(1)} \equiv 0$) on Hopf manifolds $\mathbb{S}^{2n-1}\times \mathbb{S}^{1}$. This work was extended by Wang and Yang, who gave examples of metrics on Hopf manifolds satisfying ${}^t\operatorname{Ric}^{(1)} \equiv 0$ for all $t < 1$, which includes all previously distinguished Hermitian connections, except for the Chern connection (occuring at $t = 1$) \cite[Theorem 1.11]{WangYang}.

Correa \cite{Correa} extended the Liu--Yang construction \cite{LiuYangRicci}, producing first Lichnerowicz Ricci-flat metrics on the suspension of a compact Sasaki--Einstein manifold (endowed with a Sasaki automorphism and positive constant). Recall that if $(\mathcal{Q}, g_{\mathcal{Q}})$ is a Riemannian manifold,  we say that $(\mathcal{Q}, g_{\mathcal{Q}})$ is \textit{Sasakian} if the metric cone $(\mathcal{C}(\mathcal{Q}), g_{\mathcal{C}(\mathcal{Q})})$, where $\mathcal{C}(\mathcal{Q}): = \mathcal{Q} \times \mathbb{R}_+$ and $g_{\mathcal{C}(\mathcal{Q})} : = r^2 g_{\mathcal{Q}} + dr \otimes dr$, is a K\"ahler cone. A \textit{Sasaki morphism} is an isometric immersion $(\mathcal{Q}_1,g_{\mathcal{Q}_1}) \to (\mathcal{Q}_2,g_{\mathcal{Q}_2})$ such that the induced map on the cones $\mathcal{C}(\mathcal{Q}_1) \to \mathcal{C}(\mathcal{Q}_2)$ is holomorphic. A \textit{Sasaki automorphism} is an invertible Sasaki morphism $\varphi : (\mathcal{Q}, g_{\mathcal{Q}}) \to (\mathcal{Q}, g_{\mathcal{Q}})$ with $\varphi^{-1}$ a Sasaki morphism. Given a Sasaki automorphism $\varphi$ and a real number $\kappa \in (0,\infty)\setminus\{1\}$, we define the suspension by $(\varphi,\kappa)$ of $\mathcal{Q}$, as
\[
\Sigma_{\varphi,\kappa}(\mathcal{Q}) \ := \  \frac{\mathcal{Q}\times [0,\log \kappa]}{(\varphi(x),0)\sim (x,\log \kappa)}.
\]
We extend the results of both Correa and Wang--Yang, obtaining the following:

\begin{thm}\label{MainThm4}
Let $(\mathcal{Q},g_{\mathcal{Q}})$ be a compact Sasaki--Einstein manifold. Let $\Phi : \mathcal{Q} \to \mathcal{Q}$ be a Sasaki automorphism and $\kappa >0$ a positive constant. Then the suspension $\Sigma_{\Phi,\kappa}(\mathcal{Q})$ admits a Hermitian metric $\omega$ such that $${}^t \text{Ric}_{\omega}^{(1)} \ = \ 0$$ for all $t \in (-\infty,1)$.  In particular,  $\Sigma_{\Phi,\kappa}(\mathcal{Q})$ supports first Bismut Ricci-flat metrics,  first Hermitian conformal Ricci-flat metrics, first Minimal Ricci-flat metrics, and first Lichnerowicz Ricci-flat metrics. 
\end{thm}

There are four distinct Gauduchon Ricci curvatures (Definition \ref{def:Ricci}). These, in turn, furnish two distinct scalar curvatures: \begin{eqnarray*}
{}^t \text{Scal}_{\omega} \ : = \ \text{tr}_{\omega} \left( {}^t \text{Ric}_{\omega}^{(1)} \right) \ = \ \text{tr}_{\omega} \left( {}^t \text{Ric}_{\omega}^{(2)} \right), \hspace*{1.5cm} {}^t \widetilde{\text{Scal}}_{\omega} \ : = \ \text{tr}_{\omega} \left( {}^t \text{Ric}_{\omega}^{(3)} \right) \ = \ \text{tr}_{\omega} \left( {}^t \text{Ric}_{\omega}^{(4)}\right).
\end{eqnarray*}

To further propagate the propaganda of \cite{BroderTangAltered}, we refer to ${}^t \text{Scal}_{\omega}$ as the \textit{Gauduchon scalar curvature} and refer to ${}^t \widetilde{\text{Scal}}_{\omega}$ as the \textit{Gauduchon altered scalar curvature} (c.f., Remark \ref{AlteredRmk}). We compute simple formulae for these curvatures in terms of ${}^c\operatorname{Scal}$ and ${}^c\widetilde{\operatorname{Scal}}$ (see Corollary \ref{ScalarcurvatureRelations}), which we can use to extend the vanishing theorem of Kobayashi--Wu \cite{KobayashiWu} to the Gauduchon setting:

\begin{prop}\label{MainThm6}
Let $(X, \omega)$ be a compact Hermitian manifold. Suppose that one of the following conditions holds: \begin{itemize}
    \item[(i)] ${}^t \text{Scal}_{\omega} + (t-1) {}^c \widetilde{\text{Scal}}_{\omega} >0$ for some $t>0$. 
    \item[(ii)] ${}^t \text{Scal}_{\omega} + (t-1) {}^c \widetilde{\text{Scal}}_{\omega}<0$ for some $t<0$.
    \item[(iii)] ${}^t \text{Scal}_{\omega} + (1-t)(d^{\ast} \tau + | \tau |^2) \ > \ 0$ for some $t \in \mathbb{R}$.
\end{itemize}
Then the Kodaira dimension $\kappa(X)=-\infty$.
\end{prop}
Here $\tau$ denotes the \emph{torsion $(1,0)-$form} defined by $\partial \omega^{n-1} = \tau \wedge \omega^{n-1}$.

We now wish to investigate the Gauduchon holomorphic sectional curvature. This has remained very poorly understood (even in the K\"ahler setting), but there has been some recent developments obtained by Chen--Nie \cite{ChenNieHSC} for compact complex surfaces. Recall that if ${}^t R$ denotes the $t$--Gauduchon curvature tensor of a Hermitian metric $\omega$, the \textit{Gauduchon holomorphic sectional curvature} ${}^t \text{HSC}_{\omega}$ in the direction of a $(1,0)$--tangent vector $v$ is given by $${}^t \text{HSC}_{\omega}(v) \ := \ \frac{1}{| v |_{\omega}^4} {}^t R(v,\overline{v}, v, \overline{v}).$$

In \cite{BroderTangAltered}, the first named author, together with Kai Tang, introduced the (Chern) altered holomorphic sectional curvature (although it does appear implicitly much earlier, see, for instance,  \cite{YangZhengRBC}). The definition given in \cite{BroderTangAltered} readily extends to the $t$--Gauduchon connection: $${}^t\widetilde{\text{HSC}}_{\omega} : \mathcal{F}_X \times \mathbb{R}^n \backslash \{ 0 \} \to \mathbb{R}, \hspace{1cm} {}^t \widetilde{\text{HSC}}_{\omega}(v) \ : = \ \frac{1}{| v |_{\omega}^2} \sum_{\alpha,\gamma} \left( {}^t R_{\alpha \bar{\alpha} \gamma \bar{\gamma}} + {}^t R_{\alpha \bar{\gamma} \gamma \bar{\alpha}} \right)v_{\alpha} v_{\gamma},$$ where $\mathcal{F}_X$ denotes the unitary frame bundle and $v = (v_1, ..., v_n) \in \mathbb{R}^n \backslash \{ 0 \}$.  The altered holomorphic sectional curvature ${}^t \widetilde{\text{HSC}}$ is comparable to the familiar holomorphic sectional curvature ${}^t \text{HSC}$ in the sense that they always have the same sign.  The altered holomorphic sectional curvature provides an interpretation of the familiar holomorphic sectional curvature in terms of a quadratic form-valued function on the unitary frame bundle (for more in this direction, together with details of the program initiated by the first named author to study the growing wilderness of curvatures in Hermitian geometry, see \cite{BroderSBC1, BroderSBC2, BroderLA, BroderGraph, BroderQOBC,BroderTangAltered}) and is typically easier to work with (in comparison with the holomorphic sectional curvature), at the expense of frame-independence. \\

The main theorem concerning the Gauduchon holomorphic sectional curvature is the following monotonicity theorem:

\begin{thm}\label{MainHSCThm1}
Let $(X, \omega)$ be a Hermitian manifold.  For any local unitary frame, the Gauduchon altered holomorphic sectional curvature is given by \begin{eqnarray}
{}^t \widetilde{\text{HSC}}_{\omega}(\lambda) &=& {}^c \widetilde{\text{HSC}}_{\omega}(\lambda) - \frac{(t-1)^2}{4 | \lambda |_{\omega}^2} \sum_{i,k,q} \left( {}^c T_{iq}^i \overline{{}^c T_{kq}^k} + {}^c T_{iq}^k \overline{{}^c T_{iq}^k} \right)\lambda_i \lambda_k,
\end{eqnarray}

where ${}^c T$ is the torsion of the Chern connection, and $\lambda = (\lambda_1, ..., \lambda_n) \in \mathbb{R}^n \backslash \{ 0 \}$. In particular, ${}^t \widetilde{\text{HSC}}_{\omega}  \leq {}^c \widetilde{\text{HSC}}_{\omega}$ for all $t \in \mathbb{R}$ and equality holds if and only if $t=1$ or the metric is K\"ahler.
\end{thm}

Since the altered holomorphic sectional curvature is comparable to the holomorphic sectional curvature,  the following useful consequences of the above monotonicity result are easily obtained: 

\begin{cor}\label{MainThmHSC3}
Let $(X, \omega)$ be a Hermitian manifold.  \begin{itemize}
\item[(i)] If ${}^c \text{HSC}_{\omega} \leq 0$, then ${}^t \text{HSC}_{\omega} \leq 0$ for all $t \in \mathbb{R}$.
\item[(ii)] If ${}^t \text{HSC}_{\omega}>0$ for some $t \in \mathbb{R}$,  then ${}^c \text{HSC}_{\omega}>0$.
\end{itemize}
\end{cor}

In particular, negative (respectively, positive) Chern holomorphic sectional curvature is the strongest (respectively, weakest) condition on the Gauduchon holomorphic sectional curvatures. Recall that if $(X, \omega)$ is a Hermitian manifold (not necessarily complete) with $${}^c \text{HSC}_{\omega} \leq - \kappa_0<0,$$ then the Schwarz lemma implies that $X$ is Brody hyperbolic (i.e., every holomorphic map $\mathbb{C} \to X$ is constant). In fact, by the result of Greene--Wu \cite{GreeneWu}, it suffices to assume the (Chern) holomorphic sectional curvature is bounded above by $-C(1+r^2)^{-1}$, where $r$ is the distance from a fixed point, and $C>0$ is a positive constant. 

\begin{rmk} 
For a long time, it was conjectured that every compact Kobayashi hyperbolic manifold supports a Hermitian metric with negative Chern holomorphic sectional curvature. Evidence for this conjecture was given by the local constructions of Grauert--Reckziegel \cite{GrauertReckziegel} and Cowen \cite{Cowen}, together with the results on fibrations given by Cheung \cite{CheungPhD}. A counterexample to this conjecture was given by Demailly \cite{Demailly1997} (see also \cite{DiverioSurvey} for a nice exposition). The counterexample is given by a projective Kobayashi hyperbolic surface, fibered by genus $g>1$ curves over a genus $g>1$ curve, with a fiber sufficiently singular to violate Demailly's algebraic hyperbolicity criterion \cite{Demailly1997}. In light of this, the monotonicity theorem suggests that perhaps Kobayashi hyperbolicity is characterized by the negativity of some $t$--Gauduchon holomorphic sectional curvature, for some $t \in \mathbb{R} \backslash \{ 1 \}$. 
\end{rmk}

The following result appears implicitly in the recent paper of Chen--Nie \cite{ChenNieHSC} for compact Hermitian surfaces. We recall that Apostolov--Davidov--Muskarov \cite{ADM} showed that a compact Hermitian surface with pointwise constant Levi-Civita holomorphic sectional curvature or Chern holomorphic sectional curvature is K\"ahler. This extended an earlier work of Balas--Gauduchon \cite{BalasGauduchon} who showed that a compact Hermitian surface with pointwise constant Chern holomorphic sectional curvature is K\"ahler if the constant is non-positive. Chen--Zheng \cite{ChenZheng} considered Hermitian surfaces with pointwise constant Strominger--Bismut holomorphic sectional curvature. For the Strominger--Bismut connection, the classification splits into two cases: The Hermitian surface is K\"ahler or an isosceles Hopf surface\footnote{We remind the reader that a compact complex surface $X$ is said to be an \textit{isosceles Hopf surface} if $X$ is biholomorphic to the quotient of $\mathbb{C}^2 - \{ 0 \}$ by the infinite cyclic group generated by $(z_1, z_2) \mapsto (\lambda_1 z_1, \lambda_2 z_2)$, where $0 < | \lambda_1 | = | \lambda_2 | <1$.} with an admissible metric. This dichotomy extends to all Gauduchon connections, as was shown in the aforementioned paper of Chen--Nie \cite{ChenNieHSC}: A compact Hermitian surface with pointwise constant $t$--Gauduchon holomorphic sectional curvature is either K\"ahler, or an isosceles Hopf surface, in which case $t=-1$ or $t=3$. The following result appears to be the first of this kind in higher dimensions for a general Gauduchon connection:

\begin{thm}\label{MainThmHSC2}
Let $(X, \omega)$ be a Hermitian manifold.  If ${}^t \text{HSC}_{\omega} \equiv {}^s\text{HSC}_{\omega}$ then $t=s$, $t=2-s$, or $\omega$ is K\"ahler. 
\end{thm}

\begin{rmk}\label{DualityHSCRmk}
Comparing the above result with the result in \cite{ChenNieHSC}, we note that the map $s \mapsto 2-s$ maps $s=-1$ to $s=3$, and fixes $s=1$.
\end{rmk}

\hfill

\textbf{Structure of the manuscript.} The manuscript consists of four sections: \begin{itemize}
\item[(i)] The first main section concerns generalities and serves primarily as a reminder. This will include, but will not be restricted to the Gauduchon connections. As part of future work,  we will investigate the properties of non-Hermitian, complex-analytic connections, on holomorphic vector bundles. This paper begins to lay the foundation of such connections.
\item[(ii)] The second main section treats the Ricci curvature and scalar curvature of the $t$--Gauduchon connections, proving Theorems \ref{MainThm2}--\ref{MainThm4} and Proposition \ref{MainThm6}.
\item[(iii)] The last section proves the monotonicity theorem for the $t$--Gauduchon holomorphic sectional curvature  Theorems \ref{MainHSCThm1} and \ref{MainThmHSC2}.
\end{itemize}

Several questions are also posed throughout the manuscript.


\section{Generalities on Hermitian Connections}

\subsection*{2.1. Some Reminders}
Let $(M^{2n},g)$ be a smooth Riemannian manifold of (real) dimension $2n$. We denote by $T^{\mathbb{R}}M$ the (real) tangent bundle. Assume $M$ supports an integrable almost complex structure $J : T^{\mathbb{R}}M \to T^{\mathbb{R}}M$ compatible with the Riemannian metric $g$ in the sense that $$g(Ju,Jv) \ = \ g(u,v) \hspace{1cm} \forall u,v \in T^{\mathbb{R}}M.$$ We extend both $g$ and $J$ linearly over $\mathbb{C}$ to the complex vector bundle $T^{\mathbb{C}}M : = T^{\mathbb{R}}M \otimes \mathbb{C}$. The map $u \mapsto u - \sqrt{-1} Ju$ defines an isomorphism of (smooth, real) vector bundles $T^{\mathbb{R}} M \simeq T^{1,0}M$ onto the eigenbundle of (the complexification of) $J$ corresponding to the eigenvalue $\sqrt{-1}$.

\begin{rmk}
We will maintain the convention of referring to the real $(1,1)$--form $\omega(\cdot, \cdot) = \omega_g(\cdot, \cdot) : = g(J \cdot, \cdot)$ as the metric. Further, we use the notation $\omega_g$ to emphasize the fact that the underlying metric is $g$. If there is no need to discuss the metric $g$ explicitly, we will omit the subscript and simply write $\omega$.
\end{rmk}

Let $h : T^{\mathbb{C}}M \times T^{\mathbb{C}}M \to \mathbb{C}$ be the Hermitian form given by the complex-linear extension of $g$. Write ${}^{\text{LC}}\nabla$ for the Levi-Civita connection on $T^{\mathbb{R}}M$. The complex-linear extension of ${}^{\text{LC}} \nabla$ preserves $h$ in the sense that ${}^{\text{LC}} \nabla h =0$. It is not true, however, that ${}^{\text{LC}} \nabla J=0$. Indeed, it is well-known that this equality is equivalent to the Hermitian structure $(g,J)$ being K\"ahler. Hence, the Levi-Civita is not the most germane connection to measure the underlying complex geometry when $(g,J)$ is non-K\"ahler.

Let $\mathcal{E} \to X$ be a complex vector bundle over a complex manifold $X$. We will always understand $M$ to be a smooth manifold, and $X$ to be a complex manifold. A first-order complex-linear differential operator $\bar{\partial}^{\mathcal{E}} : \mathcal{C}^{\infty}(\mathcal{E}) \longrightarrow \Omega_X^{0,1} \otimes \mathcal{C}^{\infty}(\mathcal{E})$, acting on smooth sections of $\mathcal{E}$, is said to be a \textit{CR--operator} if it satisfies the following incarnation of the Leibniz rule: \begin{eqnarray*}
\bar{\partial}^{\mathcal{E}}(f \sigma) &=& \bar{\partial} f \otimes \sigma + f \bar{\partial}^{\mathcal{E}} \sigma, 
\end{eqnarray*}

where $f \in \mathcal{C}^{\infty}(X, \mathbb{C})$ is a smooth function, $\sigma \in \mathcal{C}^{\infty}(\mathcal{E})$ is a smooth section, and $\bar{\partial} : \mathcal{C}^{\infty}(X, \mathbb{C}) \to \Omega_X^{0,1}$ is the usual CR--operator (or Dolbeault operator) acting on functions. The $(0,1)$--part of any complex-linear connection $\nabla$ on $\mathcal{E} \to X$ defines a CR--operator. Conversely, if we fix a metric $h$ on $\mathcal{E}$ and require that $\nabla h=0$, then there is a unique CR--operator $\bar{\partial}^{\mathcal{E}}$ such that $\nabla^{0,1} = \bar{\partial}^{\mathcal{E}}$. We say that a CR--operator $\bar{\partial}^{\mathcal{E}}$ on a complex vector bundle is \textit{integrable} if $\bar{\partial}^{\mathcal{E}} \circ \bar{\partial}^{\mathcal{E}}=0$. By the famous Koszul--Malgrange theorem \cite{KoszulMalgrange}, a complex vector bundle is holomorphic if and only if it admits an integrable CR--operator.

\begin{defn}
Let $(\mathcal{E},h) \to X$ be a Hermitian\footnote{Throughout this manuscript, a Hermitian vector bundle is always understood to mean a holomorphic vector bundle endowed with a Hermitian metric.} vector bundle with integrable CR--operator $\bar{\partial}^{\mathcal{E}}$. A complex-linear connection $\nabla$ on $\mathcal{E} \to X$ is said to be \begin{itemize}
    \item[(i)] \textit{Hermitian} if $\nabla h=0$.
    \item[(ii)] \textit{complex-analytic} if $\nabla^{0,1} = \bar{\partial}^{\mathcal{E}}$.
\end{itemize}
\end{defn}

It is well-known that there is a unique complex-analytic Hermitian connection on any Hermitian vector bundle. When $\mathcal{E} = T^{1,0}X$, we call this connection the \textit{Chern connection} and denote it by ${}^c \nabla$. Here, the reader may think that the distinction between Hermitian and complex-analytic connections is for the birds and that only Hermitian connections (or Hermitian complex-analytic connections) are worthy of consideration.  Both classes of connections have appeared in the literature, however. Of course, the Gauduchon connections we will consider, in general,  are Hermitian, but not complex-analytic. Connections that are not metric have appeared in Yang--Mills theory and conformal geometry (in particular, in locally conformally K\"ahler geometry). Most notably the Weyl connection \cite{CalderbankPedersen, DIU, GauduchonEinsteinWeyl,OrneaVerbitsky}. In future work, we intend to investigate the geometry of complex-analytic non-Hermitian connections.

\subsection*{2.2. The Gauduchon Connections}
Let $(X, \omega_g)$ be a Hermitian manifold endowed with a Hermitian connection $\nabla$ on $T^{1,0}X$. The torsion $T = {}^{\nabla} T \in \Omega_X^2(TX)$ of $\nabla$ is a $2$--form with values in the (complexified) tangent bundle $TX : = T^{\mathbb{C}} X \simeq T^{1,0}X \oplus T^{0,1}X$ and determines the Hermitian connection uniquely.  It was discovered by Libermann \cite{Libermann1} and later expounded upon by Gauduchon \cite{GauduchonHermitianConnections} that the Hermitian connection is determined entirely by the $(1,1)$--part $T^{1,1}$ of the torsion. In particular, the space of Hermitian connections is modeled on the infinite-dimensional affine subspace $\Omega_X^{1,1}(TX) \hookrightarrow \Omega_X^2(TX)$.  \\

The space $\Omega_X^2(TX) := TX \otimes \Lambda^2TX$ affords a type decomposition $$\Omega_X^2(TX) \ \simeq \ \Omega_X^{2,0}(TX) \oplus \Omega_X^{1,1}(TX) \oplus \Omega_X^{0,2}(TX),$$ where
\begin{eqnarray*}
    \Omega^{2,0}(TX) &:=& T^{1,0}X \otimes \Lambda^{2,0}X \,\oplus\, T^{0,1}X \otimes \Lambda^{0,2}X,\\
    \Omega^{1,1}(TX) &:=& TX\otimes \Lambda^{1,1}X,\\
    \Omega^{0,2}(TX) &:=& T^{1,0}X \otimes \Lambda^{0,2}X\,\oplus\, T^{0,1}X \otimes \Lambda^{2,0}X.
\end{eqnarray*}

Therefore,  $T \in \Omega_X^2(TX)$ can be written $T = T^{2,0} + T^{1,1} + T^{0,2}$.  The $(0,2)$--part $T^{0,2}$ is independent of the connection and recovers the Nijenhuis tensor $\mathcal{N}^J$.  Indeed,  for $(0,1)$--tangent vector fields $u : = u_0 + \sqrt{-1} J u_0$ and $v = v_0 + \sqrt{-1} J v_0$ such that $\nabla u_0 = \nabla v_0 =0$ at the point where we compute $T$, we have \begin{eqnarray*}
T^{0,2}(u,v) &=&  T(u_0, v_0) + \sqrt{-1}   T(u_0,Jv_0) + \sqrt{-1}  T(Ju_0, v_0) -  T(Ju_0, Jv_0) \\
&=& - [u_0, v_0] - J[u_0, Jv_0] + J[Ju_0, v_0] + [Ju_0, Jv_0] \ = \ \mathcal{N}^J(u_0,v_0).
\end{eqnarray*}

Since the complex structure $J$ is integrable, we see that the $(0,2)$--part $T^{0,2}$ of the torsion vanishes identically.  It therefore remains to understand the $T^{2,0}$ and $T^{1,1}$ components of $T$.  \\

Let $\Omega_X^{3,+}$ denote the space of $3$--forms of type $(2,1)+(1,2)$.  In particular, if $\alpha \in \Omega_X^3$ is a $3$--form, then we write $\alpha^+ : = \alpha^{(2,1)} + \alpha^{(1,2)} \in \Omega_X^{3,+}$ for the $((2,1) + (1,2))$--part of $\alpha$.  There is an isomorphism between $\Omega_X^{3,+}$ and $\Omega_X^{2,0}(TX)$ realized by the \textit{Bianchi projector} $\mathcal{B} : \Omega_X^2(TX) \to \Omega_X^3$ specified by the formula \begin{eqnarray*}
\mathcal{B}(\alpha) (u,v,w) \ := \ \frac{1}{3}\left(g(\alpha(u,v),w) + g(\alpha(v,w),u) + g(\alpha(w,u),v)\right),
\end{eqnarray*}
for $u,v,w \in TX$ and $\alpha \in \Omega_X^2(TX)$.  Let $\Omega_b^{1,1}(TX) : = \ker \left( \mathcal{B} \vert_{\Omega_X^{1,1}(TX)} \right)$ denote the subspace of $TX$--valued $(1,1)$--forms $\xi \in \Omega_X^{1,1}(TX)$  which satisfy the Bianchi identity $\mathcal{B}(\xi)=0$. With respect to the metric on $\Omega_X^{1,1}(TX)$ induced by the metric $\omega_g$ on $T^{1,0}X$,  we let $\Omega_c^{1,1}(TX)$ denote the orthogonal complement of $\Omega_b^{1,1}(TX)$ inside $\Omega_X^{1,1}(TX)$. Therefore, we have a further refinement of the type decomposition: $$\Omega_X^2(TX) \ \simeq \ \Omega_X^{2,0}(TX) \oplus \Omega_b^{1,1}(TX) \oplus \Omega_c^{1,1}(TX) \oplus \Omega_X^{0,2}(TX),$$ and a corresponding refinement of the torsion \begin{eqnarray}\label{TorsionDecomp}
T \ = \ T^{2,0} + T_b^{1,1} + T_c^{1,1}
\end{eqnarray} (recalling that $T^{0,2} = \mathcal{N}^J \equiv 0$). The Bianchi projector $\mathcal{B}$, in addition,  furnishes an isomorphism between $\Omega_c^{1,1}(TX)$ and $\Omega_X^{3,+}$ (in particular, $\Omega_X^{2,0}(TX) \simeq \Omega_X^{3,+} \simeq \Omega_c^{1,1}(TX)$) (c.f., \cite[Remark 3]{GauduchonHermitianConnections}).  Therefore,  to understand $T^{2,0}$ or $T_c^{1,1}$, it suffices to understand their images $\mathcal{B}(T^{2,0})$ or $\mathcal{B}(T_c^{1,1})$. \\

The $(2,0)$--part $T^{2,0}$ is determined by the $(1,1)$--part $T^{1,1}$ of the torsion: \begin{eqnarray}\label{dcomega}
\mathcal{B}(T^{2,0} - T_c^{1,1}) \ = \ \frac{1}{3} d^c \omega.
\end{eqnarray} Indeed,  for $(1,0)$--tangent vectors $u,v,w \in T^{1,0}X$,  the definition of $T$ and $\mathcal{B}$ informs us that the $(2,1)$--part of $3 \mathcal{B}(T^{2,0}-T_c^{1,1})$ is given by \begin{eqnarray*}
3 \mathcal{B}(T^{2,0} - T_c^{1,1})^{(2,1)}(u,v,\bar{w}) &=& g(\bar{w}, T^{2,0}(u,v)) - g(u, T_c^{1,1}(v, \bar{w})) - g(v, T_c^{1,1}(\bar{w},u)) \\
&=& u \cdot g(v, \bar{w}) - v \cdot g(u,\bar{w}) \\
&& \hspace*{2cm} - g([u,v], \bar{w}) + g(u, [v,\bar{w}]) + g(v, [\bar{w},u]),
\end{eqnarray*}

where the last equality uses the compatibility of $\nabla$ with the metric.  Expressing this with the associated $(1,1)$--form $\omega$,  the standard formula for the exterior derivative yields $3 \mathcal{B}(T^{2,0} - T_c^{1,1})^{(2,1)}(u,v,\bar{w}) = - \sqrt{-1} \partial \omega(u,v,\bar{w})$. By conjugating, we see that $$\mathcal{B}(T^{2,0} - T_c^{1,1}) \ = \ \frac{\sqrt{-1}}{3} (\bar{\partial} - \partial)\omega \ = \ \frac{1}{3} d^c \omega.$$

Hence, the torsion of a Hermitian connection (and hence, the Hermitian connection itself) is completely determined from the $(1,1)$--part $T^{1,1} \in \Omega_X^{1,1}(TX)$.  Using the Bianchi projector $\mathcal{B}$, one can break this up into a prescription of $T_b^{1,1}$ and $T_c^{1,1}$.  Indeed,  if $T_b^{1,1} = \sigma$ and $T_c^{1,1} = \mu$, then $$T \ = \ T^{2,0} + T_b^{1,1} + T_c^{1,1} \ = \  2 \mu + \frac{1}{3} \mathcal{B}^{-1} \left( d^c \omega \right) + \sigma.$$

\begin{defn}
Let $(X, \omega)$ be a Hermitian manifold.  A Hermitian connection $\nabla$ on $T^{1,0}X$ is said to be \textit{Gauduchon} (with \textit{Gauduchon parameter} $t$) if the torsion $T = {}^{\nabla} T$ satisfies $$T_b^{1,1} \ = \ 0, \hspace*{1cm} \mathcal{B}(T_c^{1,1}) \ = \ \frac{(t-1)}{3} d^c \omega.$$ 
\end{defn}

By construction, the Gauduchon connections form a one-dimensional affine subspace of the space of Hermitian connections. Moreover, the declaration $T_b^{1,1}=0$ is motivated by the absence of any bona fide representative of $\Omega_b^{1,1}(TX)$. On the other hand, one can consider more generally, $\mathcal{B}(T_c^{1,1}) = a\,d\omega + b\,d^c \omega$, for $a,b \in \mathbb{R}$. This recovers the \textit{Ehresmann--Libermann plane} \cite{Libermann1}.


\begin{ex}
For $t = 1$, we recover the Chern connection, and for $t=0$, we recover the Lichnerowicz connection. For $t=-1$, we recover the \textit{Strominger--Bismut connection} ${}^b \nabla$ discovered in the physics literature by Strominger \cite{Strominger1986} and later, independently in the mathematics literature by Bismut \cite{Bismut}, characterized by its torsion being totally skew-symmetric. The Hermitian conformal connection \cite{Libermann1}, namely, the Hermitian connection for which the torsion satisfies the Bianchi identity, is recovered from $t=\frac{1}{2}$. Moreover, the \textit{Minimal connection} \cite{GauduchonHermitianConnections}, defined to be the unique Hermitian connection with smallest pointwise norm of its torsion, manifests from $t = \frac{1}{3}$.
\end{ex}

\begin{rmk}\label{ZhaoZhengDualityRmk}
The presence of a numerical value attached to a connection has led to some curious `dualities' in the behavior of these Gauduchon connections. For instance, it was observed by Zhao--Zheng \cite{ZhaoZheng} that the Gauduchon connections come in dual pairs, with the duality map specified by $\mathbb{R} \backslash \{ \frac{1}{2} \} \ni t \mapsto \frac{t}{2t-1}$. This map has fixed points $t=1$ (Chern) and $t=0$ (Lichnerowicz); curiously the map is undefined for the Hermitian conformal connection ($t=1/2$). We also observe that the Strominger--Bismut connection ${}^b \nabla$ is `dual' to the minimal connection ${}^{\frac{1}{3}} \nabla$. As we discussed in Remark \ref{DualityHSCRmk}, for the holomorphic sectional curvature, there appears to be a `duality' specified by the affine map $t \mapsto 2-t$.
\end{rmk}

\subsection*{2.3 The Gauduchon curvatures}
For the remainder of this section, we focus on computing the curvature of a Hermitian connection in terms the Chern curvature and torsion on the tangent bundle. These are standard computations (seen for instance in \cite[Corollary 1.7]{WangYang}), but we perform them here for convenience of the reader. Let $\nabla$ be a Hermitian connection on $T^{1,0}X$ and denote by ${}^c\nabla$ the Chern connection, also on $T^{1,0}X$. As we observed above, $\nabla$ is completely determined by the $(1,1)$-part of its torsion. Indeed, since the torsion of ${}^c\nabla$ has vanishing $(1,1)$-part, we have 
\[
\nabla_{\overline u}v = {}^c\nabla_{\overline u}v + T(\overline u,v)^{1,0}, \qquad u,v \in T^{1,0}X,
\]
and the other components are determined from the Hermitian condition on $\nabla$. In light of this, we define the section ${}^\nabla A \in \Omega^{0,1}_X \otimes \operatorname{End}(T^{1,0}X)$ by
\[
{}^\nabla A_{\overline u}v := T(\overline u,v)^{1,0}, \qquad u,v\in T^{1,0}X.
\]
Thus, we may write
\[
\nabla = {}^c\nabla + {}^\nabla A + {}^\nabla A^*,
\]
where ${}^\nabla A^* \in \Omega^{1,0}_X \otimes \operatorname{End}(T^{1,0}X)$ denotes the $g$-adjoint of ${}^\nabla A$. We can now compute the curvature of $\nabla$ in terms of $A$, and its covariant derivatives with respect to the Chern connection:

\begin{prop}\label{prop:ACurvatureVB} Let $(X,\omega)$ be a Hermitian manifold and $\nabla$ be a Hermitian connection on $T^{1,0}X$. Then the curvature ${}^{\nabla}\Theta \in \Omega^{2}(X)\otimes \operatorname{End}(T^{1,0}X)$ satisfies:
\begin{eqnarray}
    {}^{\nabla} \Theta(u,\overline v) &=& {}^c\Theta(u,\overline v) + ({}^c\nabla_uA)_{\overline v} + (({}^c\nabla_{v}A)_{\overline u})^* - [A_u^*,A_{\overline v}], \label{eqn1} \\
    {}^{\nabla} \Theta(u,v) &=& [(A_{\overline u})^*,(A_{\overline u})^*] + (({}^c\nabla_{\overline v}A)_{\overline u})^* - (({}^c\nabla_{\overline u}A)_{\overline v})^* -A_{{}^cT(u,v)}^*, \label{eqn2}
\end{eqnarray}
for all $u,v \in T^{1,0}X$, where $A = {}^\nabla A$.
\end{prop}
\begin{proof}
It suffices to prove the assertion at a point $p \in X$. Let $u,v,w \in T^{1,0}_pX$. Extend $u,v$ and $w$ to $(1,0)$ vector fields such that ${}^c\nabla u = {}^c\nabla v = {}^c\nabla w = 0$, at $p \in X$. Hence, at the point $p$, we have
\begin{eqnarray*}
{}^{\nabla} \Theta(u,\overline v)w &=& \left({}^c\nabla_u - (A_{\overline u})^*\right)\left({}^c \nabla_{\overline v}w + A_{\overline v}w\right) - \left({}^c \nabla_{\overline v} + A_{\overline v}\right)\left({}^c\nabla_uw - (A_{\overline u})^*w\right)\\
&=& {}^c\Theta(u,\overline v)w + {}^c(\nabla_uA)_{\overline v}w - (A_{\overline u})^*A_{\overline v}w + (({}^c\nabla_{v}A)_{\overline u})^*w + A_{\overline v}(A_{\overline u})^*w,
\end{eqnarray*}
which implies \eqref{eqn1}. For \eqref{eqn2}, we first note that ${}^cT(u,v) = -[u,v]$. Therefore,
\begin{eqnarray*}
{}^{\nabla} \Theta(u,v)w &=& \left({}^c\nabla_u - (A_{\overline u})^*\right)\left({}^c\nabla_vw - (A_{\overline v})^*w\right) - \left({}^c\nabla_v - (A_{\overline v})^*\right)\left({}^c\nabla_uw - (A_{\overline u})^*w\right) - A_{\overline{T(u,v)}}^*w\\
&=& -(({}^c\nabla_{\overline u}A)_{\overline v})^*w + (A_{\overline v}A_{\overline u})^*w  + ((\nabla_{\overline v}A)_{\overline u})^*w - (A_{\overline u}A_{\overline v})^*w- A_{\overline{T(u,v)}}^*w.
\end{eqnarray*}
\end{proof}

Let $(X, \omega_g)$ be a Hermitian manifold. In this section, we compute the Gauduchon curvature tensor ${}^t R$ in terms of the Chern curvature tensor ${}^c R$ and the torsion ${}^c T$ of the Chern connection. First, we connect to the notation of the previous discussion.
\begin{lem} \label{lem:AtGauduchon}
Let $(X, \omega_g)$ be a Hermitian manifold. For the $t$--Gauduchon connection ${}^t \nabla$ the CR--torsion ${}^tA := {}^{{}^t\nabla}A$ is given by
\[
g({}^tA_{\overline u}v,\overline w) = \frac{1-t}{2}g(v,\overline{{}^cT(u,w)}),
\]
for $u,v,w \in T^{1,0}X$.
\end{lem}
\begin{proof}
Let ${}^tA$ be as claimed. We must show that ${}^t\nabla = {}^tD$ for all $t \in \mathbb{R}$, where ${}^tD$ is the Hermitian connection defined by
\[
{}^tD_{\overline u} = {}^c\nabla_{\overline u} + {}^tA_{\overline u},\qquad {}^tD_{u} = {}^c\nabla_{u} - ({}^tA_{\overline u})^*,
\]
for all $u \in T^{1,0}X$. Since the Gauduchon connections form a line, it suffices to verify this for $t = 1$ and $t = -1$ (the Chern and Bismut connections). Since ${}^1A = 0$, we have ${}^1D = {}^c\nabla = {}^1\nabla$ as required. On the other hand, it is a simple computation to show that the torsion of ${}^{-1}D$ is totally skew-symmetric. Hence ${}^{-1}D = {}^b\nabla = {}^{-1}\nabla$.
\end{proof}

Using this, we can prove the following:

\begin{thmfix}\label{MainCurvatureFormula}
Let $(X, \omega)$ be a Hermitian manifold.  In any unitary frame,  the curvature of $t$--Gauduchon connection is given by  \begin{eqnarray}
\label{eqn:tCurv11}{}^t R_{i \overline{j} k \overline{\ell}} &=& t {}^c R_{i \overline{j} k \overline{\ell}} + \frac{(1-t)}{2} \left( {}^c R_{k \overline{j} i \overline{\ell}} + {}^c R_{i \overline{\ell} k \overline{j}} \right) + \left(\frac{1-t}{2}\right)^2\left( {}^cT_{ik}^r\overline{{}^cT_{j\ell}^r}- {}^cT_{i r}^\ell\overline{{}^cT_{jr}^k} \right), \\
\label{eqn:tCurv20}{}^t R_{ij k \overline{\ell}} &=&  \frac{1-t}{2}\left({}^cT^\ell_{ik,j} - {}^cT^\ell_{jk,i}\right)+\left(\frac{1-t}{2}\right)^2\left({}^cT^r_{jk}{}^cT^\ell_{ir} - {}^cT^r_{ik}{}^cT^\ell_{jr}\right) - \frac{1-t}{2}{}^cT^r_{ij}\overline{{}^cT^l_{rk}}.
\end{eqnarray}
for $1 \leq i,j,k,\ell \leq n$.
\end{thmfix}
\begin{proof}
Let $\{e_i\}_{i=1}^n$ be a unitary frame. By Proposition \ref{prop:ACurvatureVB}, for $1 \leq i,j,k, \ell \leq n$, we have 
\begin{eqnarray}
    \label{eqn:curv11}{}^{\nabla}R_{i\overline j k \overline \ell} &=& {}^cR_{i\overline j k \overline \ell} + A_{\overline j k,i}^\ell + \overline{A_{\overline i \ell,j}^k} - \overline{A_{\overline i \ell}^r}A^r_{\overline j k} + A_{\overline j r}^{\ell}\overline{A_{\overline i r}^k},\\
    \label{eqn:curv20}{}^{\nabla} R_{ij k\overline \ell} &=& \overline{A_{\overline i \ell}^r A_{\overline j r}^k} - \overline{A_{\overline j \ell }^r A_{\overline i r}^k} + \overline{A_{\overline i \ell,\overline j}^k} - \overline{A_{\overline j \ell,\overline i}^k} - {}^cT_{ij}^r\overline{A_{\overline r l}^k},
\end{eqnarray}
where the comma indicates covariant differentiation with respect to the Chern connection ${}^c\nabla$ and $A_{\overline i j}^k := \langle A_{\overline{e_i}}e_j,\overline{e_k}\rangle$. By Lemma \ref{lem:AtGauduchon}, 
\[
    {}^t A_{\overline i j}^k \ = \ \frac{1-t}{2}\overline{{}^cT_{ik}^j}.
\]
Hence, by Equation (\ref{eqn:curv11}),
\[
{}^tR_{i\overline j k \overline \ell} = {}^cR_{i\overline j k \overline \ell} + \frac{1-t}{2}\left(\overline{{}^cT_{j\ell,\overline i}^k} + {}^cT_{ik,\overline j}^\ell\right) +\left(\frac{1-t}{2}\right)^2\left( {}^cT_{ik}^r\overline{{}^cT_{j\ell}^r}- {}^cT_{i r}^\ell\overline{{}^cT_{jr}^k} \right).
\]
The Bianchi identity ${}^c T_{ji,\overline \ell}^k = {}^cR_{i\overline \ell j \overline k} - {}^cR_{j \overline \ell i \overline k}$ for the Chern connection implies Equation (\ref{eqn:tCurv11}). On the other hand, Equation (\ref{eqn:tCurv20}) follows from Equation (\ref{eqn:curv20}).
\end{proof}

The referee pointed out to us that the above result (although proved independently) can be viewed as a Corollary of \cite[Theorem 1.6]{WangYang} (see \cite[Corollary 1.7]{WangYang}).

\section{The Gauduchon Ricci Curvatures}

We remind the reader that for a general Gauduchon connection ${}^t \nabla$, the curvature will have $(2,0)$ and $(0,2)$--parts which are non-zero.  As a consequence, there is a distinction to be made between the Ricci curvatures of a Gauduchon connection and the $(1,1)$--part of the Gauduchon Ricci curvatures. The study of the Chern Ricci curvatures and the $(1,1)$--part of the Lichnerowicz Ricci curvature was extensively studied in \cite{LiuYangRicci}. Liu--Yang also studied the Bismut Ricci curvatures in \cite{LiuYangGeometry}. We invite the reader to also see the papers \cite{Barbaro,BroderThesis,ChenChenNie,Correa,FuZhou,HeLiuYang,IvanovPapadopoulos,LiuYangLCRF,WangYang,YangZhengCurvature,YangScalarCurvature, YangKodairaDimension,YangSecondRicci,ZhaoZheng} for other known results concerning these Ricci curvatures.

\subsection*{3.1. Formulae for the Gauduchon Ricci Curvatures}

\begin{defn}\label{def:Ricci}
Let $(X, \omega_g)$ be a Hermitian manifold.  Let ${}^t \nabla$ denote the $t$--Gauduchon connection on $T^{1,0}X$ with curvature $R = {}^t R$.  We define the \textit{$t$--Gauduchon Ricci curvatures} \begin{eqnarray*}
{}^t \text{Ric}_{\omega_g}^{(1)} & : = & \sqrt{-1} {}^t \text{Ric}_{i \overline{j}}^{(1)} e^i \wedge \overline{e}^j, \hspace*{1cm} {}^t \text{Ric}_{i \overline{j}}^{(1)} \ : = \ g^{k \overline{\ell}} R_{i \overline{j} k \overline{\ell}}, \\
{}^t \text{Ric}_{\omega_g}^{(2)} & : = &  \sqrt{-1} {}^t \text{Ric}_{k \overline{\ell}}^{(2)} e^k \wedge \overline{e}^{\ell}, \hspace*{1cm} {}^t \text{Ric}_{k \overline{\ell}}^{(2)} \ : = \ g^{i \overline{j}} R_{i \overline{j} k \overline{\ell}},  \\
{}^t \text{Ric}_{\omega_g}^{(3)} & : = & \sqrt{-1} {}^t \text{Ric}_{i \overline{\ell}}^{(3)} e^i \wedge \overline{e}^{\ell}, \hspace*{1cm} {}^t \text{Ric}_{i \overline{\ell}}^{(3)} \ : = \ g^{k \overline{j}} R_{i \overline{j} k \overline{\ell}}, \\
{}^t \text{Ric}_{\omega_g}^{(4)} & : = &  \sqrt{-1} {}^t \text{Ric}_{k \overline{j}}^{(4)} e^k \wedge \overline{e}^j, \hspace*{1cm} {}^t \text{Ric}_{k \overline{j}}^{(4)} \ : = \ g^{i \overline{\ell}} R_{i \overline{j} k \overline{\ell}}.
\end{eqnarray*}
\end{defn}

By taking various traces of the formulae specified in Theorem \ref{MainCurvatureFormula}, we have:

\begin{cor}\label{RicciRelationsCor}
Let $(X, \omega)$ be a Hermitian manifold. In any unitary frame, the Gauduchon Ricci curvatures are given by \begin{eqnarray*}
{}^t \text{Ric}_{\omega}^{(1)} &=& t {}^c \text{Ric}_{\omega}^{(1)} + \frac{(1-t)}{2} \left( {}^c \text{Ric}_{\omega}^{(3)} + {}^c \text{Ric}_{\omega}^{(4)} \right) \\
{}^t \text{Ric}_{\omega}^{(2)} &=& t {}^c \text{Ric}_{\omega}^{(2)} + \frac{(1-t)}{2} \left( {}^c \text{Ric}_{\omega}^{(3)} + {}^c \text{Ric}_{\omega}^{(4)} \right) + \sqrt{-1}\frac{(1-t)^2}{4} \sum_{i,r} \left( {}^c T_{ik}^r \overline{{}^c T_{i\ell}^r} - {}^c T_{ir}^{\ell} \overline{{}^c T_{ir}^k} \right) e^k \wedge \overline{e^\ell}\\
{}^t \text{Ric}_{\omega}^{(3)} &=& t {}^c \text{Ric}_{\omega}^{(3)} + \frac{(1-t)}{2} \left( {}^c \text{Ric}_{\omega}^{(1)} + {}^c \text{Ric}_{\omega}^{(2)} \right) + \sqrt{-1}\frac{(1-t)^2}{4} \sum_{k,r} \left( {}^c T_{ik}^r \overline{{}^c T_{k\ell}^r} - {}^c T_{ir}^{\ell} \overline{{}^c T_{kr}^k} \right)e^i \wedge \overline{e^\ell} \\
{}^t \text{Ric}_{\omega}^{(4)} &=& t {}^c \text{Ric}_{\omega}^{(4)} + \frac{(1-t)}{2} \left( {}^c \text{Ric}_{\omega}^{(1)} + {}^c \text{Ric}_{\omega}^{(2)} \right) + \sqrt{-1}\frac{(1-t)^2}{4} \sum_{i,r} \left( {}^c T_{ik}^r \overline{{}^c T_{ji}^r} - {}^c T_{ir}^i \overline{{}^c T_{jr}^k} \right)e^k\wedge \overline{e^j}.
\end{eqnarray*}
\end{cor}
To simplify these expressions, we introduce the following:
\begin{notn}\label{notation3313}
    Building from \cite[p. 17]{LiuYangRicci}, we define the following $(1,1)$-forms:
    \[
        {}^c T^\diamondsuit \ := \ \sqrt{-1} {}^c T^k_{iq}\,\overline{{}^c T^k_{jq}}e^i\wedge \overline{e^j},\qquad {}^c T^\circ \ := \  \sqrt{-1} {}^c T^j_{sp}\,\overline{{}^c T^i_{sp}}e^i\wedge \overline{e^j},\qquad {}^c T^\heartsuit \ := \  \sqrt{-1} {}^c T^j_{iq}\,\overline{{}^c T^k_{kq}}e^i\wedge\overline{e^j},
    \]
    where $\{e_i\}_{i=1}^n$ is a unitary frame, $\{e^i\}_{i=1}^n$ the corresponding dual coframe, and ${}^c T_{ij}^k := g({}^cT(e_i,e_j),\overline{e_k})$ are the components of the Chern torsion in this frame. We note that ${}^c T^\diamondsuit$ and ${}^c T^\circ$ are real $(1,1)$--forms.
 \end{notn}

Recall from \cite{LiuYangRicci}, we know that ${}^c \text{Ric}_{\omega}^{(3)} = {}^c \text{Ric}_{\omega}^{(1)} - \partial \partial^{\ast} \omega$ and ${}^c \text{Ric}_{\omega}^{(4)} = {}^c \text{Ric}_{\omega}^{(1)} - \bar{\partial} \bar{\partial}^{\ast} \omega$. We may therefore write the Ricci curvatures in the following form:
\begin{thm}\label{GauduchonRicciRelations}
Let $(X, \omega)$ be a Hermitian manifold. The Gauduchon Ricci curvatures are given by \begin{eqnarray*}
{}^t \text{Ric}_{\omega}^{(1)} &=& {}^c \text{Ric}_{\omega}^{(1)} + \frac{(t-1)}{2} (\partial \partial^{\ast} \omega + \bar{\partial} \bar{\partial}^{\ast} \omega), \\
{}^t \text{Ric}_{\omega}^{(2)} &=& t {}^c \text{Ric}_{\omega}^{(2)} + (1-t) {}^c \text{Ric}_{\omega}^{(1)} + \frac{(t-1)}{2} (\partial \partial^{\ast} \omega + \bar{\partial} \bar{\partial}^{\ast} \omega) + \frac{(1-t)^2}{4}\left({}^c T^{\diamondsuit} - {}^c T^{\circ}\right), \\
{}^t \text{Ric}_{\omega}^{(3)} &=& t {}^c \text{Ric}_{\omega}^{(3)} + \frac{(1-t)}{2} \left( {}^c \text{Ric}_{\omega}^{(1)} + {}^c \text{Ric}_{\omega}^{(2)} \right) - \frac{(1-t)^2}{4}\left({}^c T^{\diamondsuit} + {}^c T^{\heartsuit} \right), \\
{}^t \text{Ric}_{\omega}^{(4)} &=& t {}^c \text{Ric}_{\omega}^{(4)} + \frac{(1-t)}{2} \left( {}^c \text{Ric}_{\omega}^{(1)} + {}^c \text{Ric}_{\omega}^{(2)} \right) - \frac{(1-t)^2}{4}\left({}^c T^{\diamondsuit} + \overline{{}^c T^{\heartsuit}}\right).
\end{eqnarray*} \end{thm}

Liu--Yang \cite{LiuYangGeometry,LiuYangRicci} showed that the second Chern Ricci curvature ${}^c \text{Ric}^{(2)}$ is related to the first Chern Ricci curvature ${}^c \text{Ric}^{(1)}$ by the formula \begin{eqnarray*}
    {}^c \text{Ric}_{\omega}^{(2)} &=& {}^c \text{Ric}_{\omega}^{(1)} - \sqrt{-1} \Lambda(\partial \bar{\partial} \omega) - (\partial \partial^{\ast} \omega + \bar{\partial} \bar{\partial}^{\ast} \omega) + {}^c T^{\circ}.
\end{eqnarray*}

In the above formula, $\Lambda$ denotes the formal adjoint of the Lefschetz operator, $\partial^{\ast}$ and $\bar{\partial}^{\ast}$ denote the formal adjoints of $\partial$ and $\bar{\partial}$. A kind exposition of these operators is given in \cite{ChenChenNie}. Let us remark that Liu--Yang \cite{LiuYangGeometry} showed that a Hermitian metric $\omega$ satisfying $\Lambda(\partial \bar{\partial} \omega)=0$ is pluriclosed (i.e., $\partial \bar{\partial}\omega=0$). From this formula for the second Chern Ricci curvature, we have the following:

\begin{cor}
Let $(X, \omega)$ be a Hermitian manifold.  The Bismut Ricci curvatures are related by \begin{eqnarray*}
{}^b \text{Ric}_{\omega}^{(1)} &=& {}^c \text{Ric}_{\omega}^{(1)} - (\partial \partial^{\ast} \omega + \bar{\partial} \bar{\partial}^{\ast} \omega), \\
{}^b \text{Ric}_{\omega}^{(2)} &=& {}^b \text{Ric}_{\omega}^{(1)} + (\partial \partial^{\ast} \omega + \bar{\partial} \bar{\partial}^{\ast} \omega) + \sqrt{-1} \Lambda(\partial \bar{\partial} \omega) + {}^cT^\diamondsuit -2 {}^c T^{\circ},\\
{}^b \text{Ric}_{\omega}^{(3)} &=& {}^b \text{Ric}_{\omega}^{(1)} + \partial \partial^{\ast} \omega - \sqrt{-1} \Lambda(\partial \bar{\partial} \omega) + {}^c T^{\heartsuit} \\
{}^b \text{Ric}_{\omega}^{(4)} &=& {}^b \text{Ric}_{\omega}^{(1)} + \bar{\partial} \bar{\partial}^{\ast} \omega - \sqrt{-1} \Lambda(\partial \bar{\partial} \omega) + \overline{{}^c T^{\heartsuit}}.
\end{eqnarray*}
\end{cor}

\begin{cor}
Let $(X, \omega)$ be a balanced manifold.  The Bismut Ricci curvatures afford the relations \begin{eqnarray*}
{}^b \text{Ric}_{\omega}^{(1)} &=& {}^c \text{Ric}_{\omega}^{(1)}  \\
{}^b \text{Ric}_{\omega}^{(2)} &=& {}^b \text{Ric}_{\omega}^{(1)} + \sqrt{-1} \Lambda(\partial \bar{\partial} \omega) +{}^cT^\diamondsuit - 2{}^c T^{\circ} \\
{}^b \text{Ric}_{\omega}^{(3)} &=& {}^b \text{Ric}_{\omega}^{(4)} \ = \  {}^b \text{Ric}_{\omega}^{(1)} - \sqrt{-1} \Lambda(\partial \bar{\partial} \omega).
\end{eqnarray*}
\end{cor} 

From the positive-definiteness of ${}^c T^{\circ}$, we have the following immediate corollary: 

\begin{cor}
Let $(X, \omega)$ be a balanced manifold.  \begin{itemize}
\item[(i)] ${}^b \text{Ric}_{\omega}^{(2)} \leq {}^b \text{Ric}_{\omega}^{(1)} + {}^cT^\diamondsuit + \sqrt{-1} \Lambda(\partial \bar{\partial} \omega)$ with equality if and only if $\omega$ is K\"ahler.
\item[(ii)] ${}^b \text{Ric}_{\omega}^{(3)} = {}^b \text{Ric}_{\omega}^{(1)}$ or ${}^b \text{Ric}_{\omega}^{(4)} = {}^b \text{Ric}_{\omega}^{(1)}$ if and only if $\omega$ is K\"ahler.
\end{itemize}
\end{cor} 

\begin{cor}\label{BalCorGauRic}
Let $(X, \omega)$ be a compact balanced manifold. Then the Gauduchon Ricci curvatures are given by \begin{eqnarray*}
{}^t \text{Ric}_{\omega}^{(1)} &=& {}^c \text{Ric}_{\omega}^{(1)}, \\
{}^t \text{Ric}_{\omega}^{(2)} &=& t {}^c \text{Ric}_{\omega}^{(2)} + (1-t) {}^c \text{Ric}_{\omega}^{(1)} + \frac{(1-t)^2}{4}\left({}^c T^{\diamondsuit} - {}^c T^{\circ}\right), \\
{}^t \text{Ric}_{\omega}^{(3)} &=& t {}^c \text{Ric}_{\omega}^{(3)} + \frac{(1-t)^2}{4} \left( {}^c \text{Ric}_{\omega}^{(1)} + {}^c \text{Ric}_{\omega}^{(2)} \right) - \frac{(1-t)^2}{4} {}^c T^{\diamondsuit}, \\
{}^t \text{Ric}_{\omega}^{(4)} &=& t {}^c \text{Ric}_{\omega}^{(4)} + \frac{(1-t)}{2} \left( {}^c \text{Ric}_{\omega}^{(1)} + {}^c \text{Ric}_{\omega}^{(2)} \right) - \frac{(1-t)^2}{4} {}^c T^{\diamondsuit}.
\end{eqnarray*}
\end{cor}

\begin{proof}
It suffices to show that if $\omega$ is balanced, then $\partial \partial^{\ast} \omega + \bar{\partial} \bar{\partial}^{\ast} \omega =0$ and ${}^c T^{\heartsuit}=0$. The first assertion is well-known (see, e.g., \cite[Lemma 2.3]{ChenChenNie}). For the latter claim,  we see that if the metric is balanced, then $${}^c T^{\heartsuit}_{k \overline{k}} \lambda_k^2 \ =\  \sum_{p,q,k} {}^c T_{kq}^k \overline{{}^c T_{qp}^p} \lambda_k^2 \ = \ \sum_{q,k} \lambda_k^2 {}^c T_{kq}^k \left( \sum_p \overline{{}^c T_{qp}^p} \right) \ =\ 0.$$
\end{proof}

Although we will focus on the scalar curvatures in a later section, it will be convenient to give the formal definitions here:

\begin{defn}
Let $(X, \omega)$ be a Hermitian manifold. The \textit{Gauduchon scalar curvature} ${}^t \text{Scal}_{\omega}$ and \textit{Gauduchon altered scalar curvature} ${}^t \widetilde{\text{Scal}}_{\omega}$ are defined by \begin{eqnarray*}
    {}^t \text{Scal}_{\omega} \ := \ \text{tr}_{\omega} \left( {}^t \text{Ric}_{\omega}^{(1)} \right) \ = \ \text{tr}_{\omega} \left( {}^t \text{Ric}_{\omega}^{(2)} \right), \qquad {}^t \widetilde{\text{Scal}}_{\omega} \ : = \ \text{tr}_{\omega} \left( {}^t \text{Ric}_{\omega}^{(3)} \right) \ = \ \text{tr}_{\omega} \left( {}^t \text{Ric}_{\omega}^{(4)} \right).
\end{eqnarray*}
\end{defn}

Immediate from \cite[Proposition 3.2]{ChenChenNie} is the following corollary: 

\begin{cor}
Let $(X, \omega)$ be a compact locally conformally K\"ahler manifold.  The first $t$--Gauduchon Ricci curvature is given by \begin{eqnarray*}
{}^t \text{Ric}_{\omega}^{(1)} &=& \frac{t(n-1) +1}{n} {}^c \text{Ric}_{\omega}^{(1)} + \frac{(t-1)(1-n)}{n} {}^c \text{Ric}_{\omega}^{(2)} + \frac{(t-1)}{n} \left( {}^c \text{Scal}_{\omega} - {}^c \widetilde{\text{Scal}}_{\omega} \right) \omega.
\end{eqnarray*}
\end{cor}

In particular, \begin{eqnarray*}
{}^b \text{Ric}_{\omega}^{(1)} &=& \frac{(2-n)}{n} {}^c \text{Ric}_{\omega}^{(1)} + \frac{2(n-1)}{n} {}^c \text{Ric}_{\omega}^{(2)} - \frac{2}{n} \left( {}^c \text{Scal}_{\omega} - {}^c \widetilde{\text{Scal}_{\omega}} \right) \omega, \\
{}^{\frac{1}{1-n}} \text{Ric}_{\omega}^{(1)} &=& {}^c \text{Ric}_{\omega}^{(2)} + \frac{1}{1-n} \left( {}^c \text{Scal}_{\omega} - {}^c \widetilde{\text{Scal}}_{\omega} \right) \omega.
\end{eqnarray*}

\begin{cor}\label{LcKScalRelation}
Let $(X, \omega)$ be a compact locally conformally K\"ahler surface. Then \begin{eqnarray*}
    {}^b \text{Ric}_{\omega}^{(1)} &=& {}^c \text{Ric}_{\omega}^{(2)} + \left( {}^c \widetilde{\text{Scal}}_{\omega} - {}^c \text{Scal}_{\omega} \right) \omega.
\end{eqnarray*}
Hence, \begin{itemize}
    \item[(i)] if ${}^b \text{Ric}_{\omega}^{(1)}={}^c \text{Ric}_{\omega}^{(2)}$, the metric is K\"ahler;
    \item[(ii)] if ${}^c \text{Ric}_{\omega}^{(2)}=0$, the metric has pointwise constant first Bismut Ricci curvature, with the (pointwise) Einstein constant ${}^c \widetilde{\text{Scal}}_{\omega} - {}^c \text{Scal}_{\omega}$;
    \item[(iii)] if ${}^b \text{Ric}_{\omega}^{(1)}=0$, the metric has pointwise constant second Chern Ricci curvature, with the (pointwise) Einstein constant ${}^c \text{Scal}_{\omega} - {}^c \widetilde{\text{Scal}}_{\omega}$.
\end{itemize}
\end{cor}

\subsection*{3.2. Applications}
Liu--Yang \cite{LiuYangRicci} gave a very systematic study of the relationship between the Chern Ricci curvature ${}^c \text{Ric}_{\omega}^{(k)}$ and the Lichnerowicz Ricci curvature ${}^0 \text{Ric}_{\omega}^{(k)}$.  Note that they refer to the Lichnerowicz Ricci curvatures as the Levi-Civita Ricci curvatures, which they denote by $\mathfrak{R}^{(k)}$. We will see that many of their results on the Lichnerowicz Ricci forms hold for many other Gauduchon connections.  \\

Let us remind the reader that a compact complex manifold $X$ is said to be in the \textit{Fujiki class $\mathcal{C}$} if $X$ is bimeromorphic to a compact K\"ahler manifold. 

\begin{cor}
Let $(X, \omega)$ be a compact Hermitian manifold.  If ${}^t \text{Ric}_{\omega}^{(1)}>0$ or ${}^t \text{Ric}_{\omega}^{(1)}<0$ for some $t \in \mathbb{R}$, then $X$ is in the Fujiki class $\mathcal{C}$ if and only if $X$ is K\"ahler. \end{cor} \begin{proof}
Theorem \ref{GauduchonRicciRelations} shows that for all $t \in \mathbb{R}$, the first Gauduchon Ricci curvature is $\partial \bar{\partial}$--closed.  Hence, if ${}^t \text{Ric}_{\omega}^{(1)}$ is definite,  we can produce a pluriclosed metric by setting $\omega_t : = {}^t \text{Ric}_{\omega}^{(1)}$ or $\omega_t :=-{}^t \text{Ric}_{\omega}^{(1)}$, depending on the sign of ${}^t \text{Ric}_{\omega}^{(1)}$.  By a theorem of Chiose \cite{Chiose},  a compact complex manifold in the Fujiki class $\mathcal{C}$ admits a pluriclosed metric if and only if it is K\"ahler.
\end{proof}

\begin{defn}
Let $X$ be a complex manifold.  \begin{itemize}
\item[(i)] The \textit{Bott--Chern cohomology groups} are defined \begin{eqnarray*}
H_{\text{BC}}^{p,q}(X) & : = & \frac{\{ \alpha \in \Omega_X^{p,q} : d \alpha =0 \}}{\{ \sqrt{-1} \partial \bar{\partial} \beta : \beta \in \Omega_X^{p-1,q-1} \}}.
\end{eqnarray*}
\item[(ii)] The \textit{Aeppli cohomology groups} are defined \begin{eqnarray*}
H_{\text{A}}^{p,q}(X) & : = &  \frac{\{ \alpha \in \Omega_X^{p,q} : \partial \bar{\partial} \alpha =0 \}}{\{ \partial \gamma + \bar{\partial} \gamma : \gamma \in \Omega_X^{p-1, q-1} \}}.
\end{eqnarray*}
\end{itemize}
\end{defn}

Recall that if $h$ is a Hermitian metric on a holomorphic line bundle $\mathcal{L} \to X$, the first Chern class $c_1(\mathcal{L})$ is represented by $\frac{\sqrt{-1}}{2\pi} \partial \bar{\partial} \log(h)$ in $H_{\text{DR}}^2(X,\mathbb{R})$.  We can similarly define the first Bott--Chern and first Aeppli--Chern classes:

\begin{defn}
Let $(\mathcal{L},h) \to X$ be a Hermitian line bundle over a complex manifold $X$.  The \textit{first Bott--Chern class} $c_1^{\text{BC}}(\mathcal{L})$ (respectively, \textit{first Aeppli--Chern class} $c_1^{\text{AC}}(\mathcal{L})$) is the cohomology class in $H_{\text{BC}}^{1,1}(X)$ (respectively, $H_A^{1,1}(X)$) represented by the curvature form $\Theta^{(\mathcal{L},h)} = \frac{\sqrt{-1}}{2\pi} \partial \bar{\partial} \log(h)$.  
\end{defn}

\begin{rmk}
The first Bott--Chern and first Aeppli--Chern classes are well-defined, independent of the specific choice of Hermitian metric.  Indeed, if $\tilde{h}$ is another Hermitian metric on $\mathcal{L} \to X$,  then the difference of the curvature forms $\Theta^{(\mathcal{L},\tilde{h})} - \Theta^{(\mathcal{L},h)} = \frac{\sqrt{-1}}{2\pi} \partial \bar{\partial} \log \left( \frac{\tilde{h}}{h} \right)$ is globally $\partial \bar{\partial}$--exact. 
\end{rmk}

\begin{rmk}
Liu--Yang \cite{LiuYangRicci} defined the first Aeppli Chern class to be the cohomology class represented by the first Lichernowicz Ricci curvature.  The above result indicates that the choice of $t=0$ is unnecessary, the first Gauduchon Ricci curvature ${}^t \text{Ric}_{\omega}^{(1)}$  represents the same cohomology class. 
\end{rmk}

From Theorem \ref{GauduchonRicciRelations}, we observe that the results in \cite{LiuYangRicci} can be generalized substantially. For instance, we have the following extension of \cite[Theorem 3.14]{LiuYangRicci}:

\begin{thm}
Let $(X, \omega)$ be a Hermitian manifold.  The first Gauduchon Ricci form ${}^t \text{Ric}_{\omega}^{(1)}$ represents $c_1^{\text{AC}}(K_X^{-1}) \in H_A^{1,1}(X)$ for all $t \in \mathbb{R}$.  Moreover, \begin{itemize}
\item[(i)] ${}^t \text{Ric}_{\omega}^{(1)}$ is $d$--closed if and only if $\partial \bar{\partial} \bar{\partial}^{\ast}\omega=0$ or $t = 1$.
\item[(ii)] If $\bar{\partial} \partial^{\ast} \omega = 0$, then ${}^t \text{Ric}_{\omega}^{(1)}$ represents the $c_1(K_X^{-1}) \in H_{\text{DR}}^2(X,\mathbb{R})$, i.e.,  $c_1(K_X^{-1})=c_1^{\text{AC}}(K_X^{-1})$.
\item[(iii)] If $\omega$ is conformally balanced, then ${}^t \text{Ric}_{\omega}^{(1)}$ represents $c_1(K_X^{-1}) \in H_{\bar{\partial}}^{1,1}(X)$ and also the first Bott--Chern class $c_1^{\text{BC}}(K_X^{-1}) \in H_{\text{BC}}^{1,1}(X)$. 
\item[(iv)] ${}^t \text{Ric}_{\omega}^{(1)} = {}^s \text{Ric}_{\omega}^{(1)}$ for $t \neq s$ if and only if $\omega$ is balanced.
\end{itemize} \end{thm} \begin{proof}
Since ${}^c \text{Ric}_{\omega}^{(1)} = -\sqrt{-1} \partial \bar{\partial} \log(\omega^n)$,  we see that \begin{eqnarray*}
d({}^t \text{Ric}_{\omega}^{(1)}) &=& \frac{(t-1)}{2} \left( \bar{\partial} \partial \partial^{\ast} \omega + \partial \bar{\partial} \bar{\partial}^{\ast} \omega \right).
\end{eqnarray*}
Decomposing the equation $d\left( {}^t \text{Ric}_{\omega}^{(1)} \right)$ into parts, proves (i).  For statement (ii),  if $\bar{\partial} \partial^{\ast} \omega =0$, then $${}^t \text{Ric}_{\omega}^{(1)} \ = \ {}^c \text{Ric}_{\omega}^{(1)} + \frac{(t-1)}{2} dd^{\ast} \omega.$$ Therefore, $\left[ {}^t \text{Ric}_{\omega}^{(1)} \right] = \left[ {}^c \text{Ric}_{\omega}^{(1)} \right]$ in $H_{\text{DR}}^2(X,\mathbb{R})$. 

Assume now that $\omega$ is conformally balanced, i.e., there is a smooth function $u$ such that $\omega_u : = e^u \omega$ is balanced.  The Christoffel symbols $\Gamma_u$ of $\omega_u$ are given in any local frame by \begin{eqnarray*}
(\Gamma_u)_{\overline{i} j}^k &=& \frac{1}{2} e^{-u} g^{k \overline{\ell}} \left( \frac{\partial}{\partial \bar{z}_i}(e^u g_{j \overline{\ell}}) - \frac{\partial}{\partial \bar{z}_{\ell}}(e^u g_{j \overline{i}}) \right)  \ = \ \Gamma_{\overline{i} j}^k + \frac{1}{2} \left( \delta_{jk} u_{\overline{i}} - g^{k \overline{\ell}} g_{j \overline{i}} u_{\overline{\ell}} \right).
\end{eqnarray*}
Therefore, \begin{eqnarray*}
(\Gamma_u)_{\overline{i} k}^k &=& \Gamma_{\overline{i} k}^k + \frac{n-1}{2} u_{\overline{i}},
\end{eqnarray*}
and hence,  $\bar{\partial}_u^{\ast} \omega_u = \bar{\partial}^{\ast} \omega + \sqrt{-1} (n-1) \partial u$.  Differentiating this expression gives \begin{eqnarray*}
\bar{\partial} \bar{\partial}_u^{\ast} \omega_u &=& \bar{\partial} \bar{\partial}^{\ast} \omega  - (n-1) \sqrt{-1} \partial \bar{\partial}u, \hspace*{1cm} \partial \bar{\partial}_u^{\ast} \omega_u \ = \ \partial \bar{\partial}^{\ast} \omega.
\end{eqnarray*}
Since $\omega_u$ is balanced,  $\bar{\partial}_u^{\ast} \omega_u=0$, and we have \begin{eqnarray*}
\partial \partial^{\ast} \omega + \bar{\partial} \bar{\partial}^{\ast} \omega &=& 2(n-1) \sqrt{-1} \partial \bar{\partial} u.
\end{eqnarray*}
In particular,  ${}^t \text{Ric}_{\omega}^{(1)} = {}^c \text{Ric}_{\omega}^{(1)} - (n-1) \sqrt{-1} \partial \bar{\partial}u$,  which proves (iii).  

If ${}^t \text{Ric}_{\omega}^{(1)} = {}^s \text{Ric}_{\omega}^{(1)}$, then \begin{eqnarray*}
\frac{(t-1)}{2} (\partial \partial^{\ast} \omega + \bar{\partial} \bar{\partial}^{\ast} \omega) &=&  \frac{(s-1)}{2} (\partial \partial^{\ast} \omega + \bar{\partial} \bar{\partial}^{\ast} \omega).
\end{eqnarray*}
Since $s \neq t$, we have $\partial \partial^{\ast} \omega=0$ and $\bar{\partial}\bar{\partial}^{\ast}\omega=0$. Hence, $\omega$ is balanced.
\end{proof}

An immediate consequence of the definition is that $$c_1^{\text{BC}}(\mathcal{L}) =0 \ \implies \ c_1(\mathcal{L}) =0 \ \implies \ c_1^{\text{AC}}(\mathcal{L}) =0.$$ If $X$ supports the $\partial \bar{\partial}$--lemma, then the converse implications hold \cite{LiuYangRicci}: 

\begin{prop}\label{ddbarChernClassProp}
Let $\mathcal{L} \to X$ be a holomorphic line bundle over a complex manifold $X$. If the $\partial \bar{\partial}$--lemma holds on $X$ (e.g., if $X$ is Moishezon or in the Fujiki class $\mathcal{C}$), then \begin{eqnarray*}
c_1^{\text{BC}}(\mathcal{L}) =0 \ \iff \ c_1(\mathcal{L}) =0 \ \iff \ c_1^{\text{AC}}(\mathcal{L}) =0.
\end{eqnarray*} \end{prop} 

\begin{ex}
Let $X = \mathbb{S}^{3} \times \mathbb{S}^1$ be the Hopf surface.  From the Kunneth formula,  the second Betti number $b_2(X)=0$, and therefore, $c_1(\mathcal{L})=c_1^{\text{AC}}(\mathcal{L})=0$ for any holomorphic line bundle $\mathcal{L} \to X$.  We claim that $c_1^{\text{BC}}(K_X^{-1}) \neq 0$. Let $\omega_0$ be the Boothby metric $$\omega_0 : = \sqrt{-1} \frac{\delta_{ij}}{| z |^2} dz_i \wedge d\overline{z}_j$$ on the Hopf surface $\mathbb{S}^3 \times \mathbb{S}^1$. The first Chern Ricci curvature is easily computed to be $${}^c \text{Ric}_{\omega_0}^{(1)} \ = \ \frac{1}{| z |^2} \left( \delta_{ij} - \frac{\overline{z}_i z_j}{| z |^2} \right) \sqrt{-1} dz_i \wedge d\overline{z}_j.$$ By the Cauchy--Schwarz inequality, ${}^c \text{Ric}_{\omega_0}^{(1)} \geq 0$.  This implies $c_1^{\text{BC}}(K_X^{-1}) \neq 0$.  Indeed, if $c_1^{\text{BC}}(K_X^{-1})=0$, then there exists a function $u$ such that ${}^c \text{Ric}_{\omega_0}^{(1)} = \sqrt{-1} \partial \bar{\partial} u \geq 0$. But this would force $u$ to be constant, which is an obvious contradiction.  These observations were made by Tosatti in \cite[Example 3.3]{TosattiNonKahlerCY}.  
\end{ex}

\begin{ex}
(\cite[Example 3.5]{TosattiNonKahlerCY}). Another example of this type is given by the hypothetical complex structure on the six-sphere $\mathbb{S}^6$.  Since $b_2(\mathbb{S}^6)=0$, it is clear that $c_1(\mathcal{L})= c_1^{\text{AC}}(\mathcal{L})=0$ for any holomorphic line bundle on $\mathbb{S}^6$. On the other hand, if $c_1^{\text{BC}}(K_X)=0$, then $K_X$ is holomorphically torsion.  Since $H^1(\mathbb{S}^6, \mathbb{Z}) = H^3(\mathbb{S}^6, \mathbb{Z}) =0$, the exponential sequence implies that $\text{Pic}(\mathbb{S}^6) \simeq H^1(\mathbb{S}^6, \mathcal{O}_{\mathbb{S}^6})$. Therefore, $K_X$ is holomorphically trivial, and there is a non-vanishing holomorphic $3$--form $\alpha$ on $\mathbb{S}^6$.  It is clear that $\alpha$ is $d$--closed, and therefore must be $d$--exact. Writing $\alpha = d\beta$, we see that \begin{eqnarray*}
0 \ < \ (\sqrt{-1})^9 \int_{\mathbb{S}^6} \alpha \wedge \overline{\alpha} \ = \ (\sqrt{-1})^9 \int_{\mathbb{S}^6} d(\beta \wedge d\overline{\beta}) \ = \ 0,
\end{eqnarray*}

an obvious contradiction.
\end{ex}

\begin{rmk}
From Proposition \ref{ddbarChernClassProp}, we deduce that the Hopf surface $\mathbb{S}^3 \times \mathbb{S}^1$ and any hypothetical complex structure on $\mathbb{S}^6$ does not support the $\partial \bar{\partial}$--lemma.
\end{rmk}

The following result was established by Correa \cite[Lemma 3.1]{Correa} in the special case of the first Lichnerowicz Ricci curvature: 

\begin{prop}\label{LcKChernClassLemma}
Let $(X,\omega)$ be a compact locally conformally K\"ahler manifold. Then $${}^t \text{Ric}_{\omega}^{(1)} \ = \ \alpha + \frac{t(1-n)-1}{2} d(J\vartheta),$$ where $\alpha$ is a representative of $2\pi c_1(K_X^{-1})$, the underlying complex structure is denoted by $J$, and $\vartheta$ denotes the Lee form. In particular, on a compact locally conformally K\"ahler manifold, if ${}^t \text{Ric}_{\omega}^{(1)}=0$ then $c_1(K_X)=0$ in $H_{\text{DR}}^2(X,\mathbb{R})$.
\end{prop} \begin{proof}
Write $\omega = e^f \omega_0$ on a suitably small open set $\mathcal{U} \subset X$, where $\omega_0$ is a K\"ahler metric defined on $\mathcal{U}$ and $f : \mathcal{U} \to \mathbb{R}$ is smooth. From the proof of \cite[Lemma 3.1]{Correa}, we know that the first Chern Ricci curvature is locally written as ${}^c \text{Ric}_{\omega}^{(1)} = - \frac{1}{2} dd^c \log(\omega_0^n) - \frac{n}{2} d(J\vartheta)$, and moreover, $\partial \partial^{\ast} \omega + \bar{\partial} \bar{\partial}^{\ast} \omega = - (n-1) d(J \vartheta)$. Hence, the first $t$--Gauduchon Ricci curvature is locally given by \begin{eqnarray*}
{}^t \text{Ric}_{\omega}^{(1)} &=& - \frac{1}{2} dd^c \log(\omega_0^n) - \frac{n}{2} d(J \vartheta) - \frac{(t-1)}{2}(n-1) d(J \vartheta) \\
&=& - \frac{1}{2} dd^c \log(\omega_0^n) + \frac{t(1-n)-1}{2} d(J\vartheta).
\end{eqnarray*}
The local expressions $- \frac{1}{2} dd^c \log(\omega_0^n)$ glue together to define a representative $\alpha$ of $2\pi c_1(K_X^{-1})$ in $H_{\text{DR}}^2(X,\mathbb{R})$. Hence, ${}^t \text{Ric}_{\omega}^{(1)}$ represents $2\pi c_1(K_X^{-1})$.
\end{proof}

The above result yields an obstruction to the existence of a locally conformally K\"ahler structure. It is well-known that obstructions to locally conformally K\"ahler structures are few and far between. In more detail, we have:

\begin{thm}\label{LcKCor}
Let $(X,\omega)$ be a compact Hermitian manifold with ${}^t \text{Ric}_{\omega}^{(1)}=0$ for some $t \in \mathbb{R}$. If $c_1(K_X) \neq 0$ in $H_{\text{DR}}^2(X,\mathbb{R})$, then $X$ is not locally conformally K\"ahler.
\end{thm}

\begin{rmk}\label{relaxAeppli}
One cannot relax the assumption of the above result to the vanishing of the first Aeppli Chern class $c_1^{\text{AC}} =0$ in $H_A^{1,1}(X)$. Indeed, it was shown by He--Liu--Yang \cite{HeLiuYang} that a compact complex surface with ${}^l \text{Ric}_{\omega}^{(1)}=0$ is minimal and is one of the following: an Enriques surface, a bi-elliptic surface, a K3 surface, a torus, or a Hopf surface. This result should be compared with Tosatti's result \cite{TosattiNonKahlerCY} compact complex surfaces with ${}^c \text{Ric}_{\omega}^{(1)}=0$. Indeed, in this case, $X$ is minimal and is one of the following: an Enriques surface, a bi-elliptic surface, a K3 surface, a torus, or a Kodaira surface. In particular, a Kodaira surface admits a first Chern Ricci-flat metric and satisfies $c_1^{\text{BC}}(K_X)=0$ (and hence, $c_1(K_X)=0$ and $c_1^{\text{AC}}(K_X)=0$), but does not support a first Lichnerowicz Ricci-flat metric.
\end{rmk}

With Corollary \ref{LcKCor} in mind, we pose the following\footnote{The first named author would like to thank Xiaokui Yang for valuable communications regarding this question.}:

\begin{q}
Do there exist compact complex manifolds with $c_1^{\text{AC}}(K_X)=0$ but $c_1(K_X) \neq 0$?
\end{q}

\subsection*{3.3. Gauduchon Ricci-flat Metrics}

Liu--Yang \cite{LiuYangRicci} constructed first Lichnerowicz Ricci-flat metrics on the Hopf manifolds $\mathbb{S}^{2n-1} \times \mathbb{S}^1$. Their construction was extended for all $t$-Gauduchon metrics with $t < 1$ by Wang--Yang \cite{WangYang}, showing that the Hopf manifolds support first Gauduchon Ricci-flat metrics for all $t \in \mathbb{R} \backslash \{ 1 \}$. Before proving Theoren \ref{MainThm4}, we recall the proof of \cite[Theorem 1.11]{WangYang}. Recall that the universal cover of the Hopf manifold $\mathbb{S}^{2n-1} \times \mathbb{S}^1$ is $\mathbb{C}^n \backslash \{ 0 \}$. In particular, the Euclidean metric on $\mathbb{C}^n \backslash \{ 0 \}$ induces a metric on $\mathbb{S}^{2n-1} \times \mathbb{S}^1$ which we refer to as the \textit{standard metric} (or \textit{Boothby metric}). In local coordinates $(z_1, ..., z_n)$, the standard metric is given by $$\omega_0 \ : = \ 4 \sqrt{-1} \sum_k \frac{dz_k \wedge d\overline{z}_k}{| z |^2}.$$ 

\begin{thmfix}[{\cite[Theorem 1.11]{WangYang}}]
Let $(X, \omega_0)$ be the Hopf manifold $X = \mathbb{S}^{2n-1} \times \mathbb{S}^1$ endowed with the standard metric $\omega_0$.  The $(1,1)$--form $$\omega \ : = \ \omega_0 + \frac{4t(1-n)-4}{n} {}^l \text{Ric}_{\omega}^{(1)}$$ is a solution of the equation $${}^t \text{Ric}_{\omega}^{(1)} \ = \ 0, \hspace{1cm} \text{for all} \  t \in \mathbb{R} \backslash \{ 1 \}.$$ Moreover,  for all $t < 1$, the $(1,1)$--form $\omega$ is positive-definite and thus defines a first Gauduchon Ricci-flat metric.
\end{thmfix}

\begin{proof}
Let $\omega_0$ denote the standard metric on the Hopf manifold $\mathbb{S}^{2n-1} \times \mathbb{S}^1$.  Following \cite[Section 6.1]{LiuYangRicci}, we set \begin{eqnarray*}
\omega_{\lambda} & : = & \omega_0 + 4 \lambda {}^l \text{Ric}_{\omega_0}^{(1)},
\end{eqnarray*}

for some $\lambda > -1$.  Introduce the notation $\alpha : = \sqrt{-1} \partial \bar{\partial} \log | z |^2$. From \cite[p. 31]{LiuYangRicci}, we know that ${}^c \text{Ric}_{\omega_{\lambda}}^{(1)} = n \alpha$ and $\partial \partial_{\lambda}^{\ast} \omega_{\lambda} + \bar{\partial} \bar{\partial}_{\lambda}^{\ast} \omega_{\lambda} = \frac{2(n-1)}{1+\lambda} \alpha$. Hence,  \begin{eqnarray*}
{}^t \text{Ric}_{\omega_{\lambda}}^{(1)} &=& {}^c \text{Ric}_{\omega_{\lambda}}^{(1)} + \frac{(t-1)}{2} \left( \partial \partial_{\lambda}^{\ast} \omega_{\lambda} + \bar{\partial} \bar{\partial}_{\lambda}^{\ast} \omega_{\lambda} \right) \\
&=& n \alpha + \frac{(t-1)}{2}  \frac{2(n-1)}{1+\lambda} \alpha \ = \ \frac{n(1+\lambda) + (t-1)(n-1)}{1+\lambda} \alpha.
\end{eqnarray*}

The coefficient of $\alpha$ vanishes for $\lambda = \frac{t(1-n)-1}{n}$.  The inequality $\lambda >-1$ (which ensures $\omega_{\lambda}$ is positive-definite) translates to $t<1$.
\end{proof}

\begin{rmk}
Observe that since $c_1^{\text{BC}}(K_X^{-1}) \neq 0$ for $X = \mathbb{S}^{2n-1} \times \mathbb{S}^1$,  there are no first Chern Ricci-flat metrics on $X$. After coming to this result independently, we are grateful to the referee our attention the result of Wang--Yang \cite[Theorem 1.11]{WangYang}.
\end{rmk}

To the author's knowledge, there is no known obstruction to the existence of first $t$--Gauduchon Ricci-flat metric metrics for $t>1$. Hence, it is natural to ask:

\begin{q}
Do there exist first $t$--Gauduchon Ricci-flat metrics on $\mathbb{S}^{2n-1} \times \mathbb{S}^1$ for $t>1$?
\end{q}

Recently, Correa \cite{Correa} extended the Liu--Yang construction to a much more general class of manifolds.  We will extend Correa's results to the entire Gauduchon line.  We first remind the reader of some terminology:

\begin{defn}
Let $(\mathcal{Q}, g_{\mathcal{Q}})$ be a Riemannian manifold.  We say that $(\mathcal{Q}, g_{\mathcal{Q}})$ is \textit{Sasakian} if the metric cone $(\mathcal{C}(\mathcal{Q}), g_{\mathcal{C}(\mathcal{Q})})$, where $\mathcal{C}(\mathcal{Q}): = \mathcal{Q} \times \mathbb{R}_+$ and $g_{\mathcal{C}(\mathcal{Q})} : = r^2 g_{\mathcal{Q}} + dr \otimes dr$, is a K\"ahler cone. A \textit{Sasaki morphism} is an isometric immersion $(\mathcal{Q}_1,g_{\mathcal{Q}_1}) \to (\mathcal{Q}_2,g_{\mathcal{Q}_2})$ such that the induced map on the cones $\mathcal{C}(\mathcal{Q}_1) \to \mathcal{C}(\mathcal{Q}_2)$ is holomorphic. A \textit{Sasaki automorphism} is an invertible Sasaki morphism $\varphi : (\mathcal{Q}, g_{\mathcal{Q}}) \to (\mathcal{Q}, g_{\mathcal{Q}})$ with $\varphi^{-1}$ a Sasaki morphism. Given a Sasaki automorphism $\varphi$ and $\kappa \in (0,\infty)\setminus\{1\}$, we define the suspension by $(\varphi,\kappa)$ of $\mathcal{Q}$, as
\[
\Sigma_{\varphi,\kappa}(\mathcal{Q}) := \frac{\mathcal{Q}\times [0,\log \kappa ]}{(\varphi(x),0)\sim (x,\log \kappa)}.
\]
\end{defn}

In the above definition, $r$ is the coordinate on the positive real line $\mathbb{R}_+ : = (0,\infty)$.  Sasaki manifolds are the odd-dimensional cousins of K\"ahler manifolds (analogous to the relationship between contact and symplectic manifolds or CR and complex manifolds). An excellent concise account of Sasaki manifolds is given in \cite[$\S$ 6.4]{OrneaVerbitsky}, some of which we borrow here.

\begin{ex}
For $n \in \mathbb{N}$, the sphere $\mathbb{S}^{2n-1}$ supports a number of Sasaki structures. Indeed, the cone $\mathcal{C}(\mathbb{S}^{2n-1}) = \mathbb{C}^n - \{ 0 \}$. More generally, Blair \cite{Blair} showed that totally umbilical, oriented real hypersurfaces in K\"ahler manifolds support a Sasaki structure. For instance, $\mathbb{S}^2 \times \mathbb{S}^3$, viewed as the unit tangent bundle of $\mathbb{S}^3$ is a Sasaki manifold. Further examples are discussed in \cite{BoG1}.
\end{ex}

\begin{defn}
Let $(\mathcal{Q},g_{\mathcal{Q}})$ be a Sasaki manifold. We say that $(\mathcal{Q},g_{\mathcal{Q}})$ is \textit{Sasaki--Einstein} if $\text{Ric}_{g_{\mathcal{Q}}} = \lambda g_{\mathcal{Q}}$, for some constant $\lambda \in \mathbb{R}$.
\end{defn}

\begin{thmfix}
Let $(\mathcal{Q},g_{\mathcal{Q}})$ be a Sasaki manifold of (real) dimension $2n+1$. The following are equivalent: \begin{itemize}
    \item[(i)] $(\mathcal{Q},g_{\mathcal{Q}})$ is Sasaki--Einstein with $\text{Ric}_{g_{\mathcal{Q}}} = 2n g_{\mathcal{Q}}$;
    \item[(ii)] The cone $(\mathcal{C}(\mathcal{Q}), g_{\mathcal{C(Q)}})$ is Ricci-flat K\"ahler.
\end{itemize}
\end{thmfix}

\begin{rmk}
By the results of \cite{BoyerGalicki1999,Kashiwada}, every Sasaki manifold of (real) dimension $3$ is Sasaki--Einstein.
\end{rmk}

\begin{ex}
Let $\mathcal{Q} = \mathbb{S}^{2n-1}$ and take the Sasaki automorphism to be the identity map $\Phi = \text{id}$. Then we recover the Hopf manifold $\Sigma_{\Phi,\kappa}(\mathcal{Q}) \ = \ \mathbb{S}^1 \times \mathbb{S}^{2n-1}$, for any $\kappa>0, \kappa \neq 1$.
\end{ex}

Let us now give prove the existence of first $t$--Gauduchon Ricci-flat metrics on certain Sasaki manifolds, extending the results of Correa \cite{Correa} (in the special case of the Lichnerowicz connection):

\begin{thmfix}
Let $(\mathcal{Q},g_{\mathcal{Q}})$ be a compact Sasaki--Einstein manifold. Let $\Phi : (\mathcal{Q},\eta_{\mathcal{Q}}) \to (\mathcal{Q},\eta_{\mathcal{Q}})$ be a Sasaki automorhism and $\kappa>0, \kappa \neq 1$ a constant. For any $t \in (-\infty, 1)$, the suspension $\Sigma_{\Phi,\kappa}(\mathcal{Q})$ admits a Hermitian metric $\omega$ such that $${}^t \text{Ric}_{\omega}^{(1)} \ = \ 0.$$ In particular, $\Sigma_{\Phi,\kappa}(\mathcal{Q})$ admits first Bimsut Ricci-flat metrics, first Hermitian conformal Ricci-flat metrics, and first Lichnerowicz Ricci-flat metrics.
\end{thmfix} \begin{proof}
Let $(\mathcal{Q},g_{\mathcal{Q}})$ be a compact Sasaki--Einstein manifold with Sasaki automorphism $\Phi$ and positive constant $\kappa$. The map $\varphi := \log(r)$ defines a diffeomorphism $\mathcal{C}(\mathcal{Q}) \simeq \mathcal{Q} \times \mathbb{R}$, and the metric $g_{\mathcal{C}(\mathcal{Q})}$ on $\mathcal{C}(\mathcal{Q})$ can be written as $g_{\mathcal{C}(\mathcal{Q})} = e^{2\varphi}(\eta_{\mathcal{Q}} + d\varphi \otimes d\varphi)$. Let $\Gamma_{\Phi, \kappa}$ be the cyclic group generated by $(x,\varphi) \mapsto (\Phi(x), \varphi + \log(\kappa))$. Endow $\mathcal{C}(\mathcal{Q})$ with the Hermitian structure $(e^{-2\varphi}g_{\mathcal{C}(\mathcal{Q})}, \mathcal{J})$. The group $\Gamma_{\Phi,\kappa}$ acts on $(\mathcal{C}(\mathcal{Q}),e^{-2\varphi}g_{\mathcal{C}(\mathcal{Q})},\mathcal{J})$ by holomorphic isometries. Hence, the quotient $\Sigma_{\Phi,\kappa}(\mathcal{Q}) : = \mathcal{C}(\mathcal{Q}) / \Gamma_{\Phi, \kappa}$ is furnished with a Hermitian structure. From \cite{GiniOrneaParton}, the Hermitian structure on $\Sigma_{\Phi,\kappa}(\mathcal{Q})$ is Vaisman (i.e., the Lee form $\vartheta$ is parallel with respect to the Levi-Civita connection). From \ref{LcKChernClassLemma}, we can write $${}^t \text{Ric}_{\omega}^{(1)} = \alpha + \frac{t(1-n)-1}{2} d (J \vartheta),$$ where $\alpha$ is a representative of $2\pi c_1(K_X^{-1})$. Let $p : \mathcal{C}(\mathcal{Q}) \to \Sigma_{\Phi, \kappa}(\mathcal{Q})$ denote the projection map. Since $\omega$ is locally written as $\omega = e^{-2\varphi} \omega_{\text{RFK}}$, where $\omega_{\text{RFK}}$ is the Ricci-flat K\"ahler metric on $\mathcal{C}(\mathcal{Q})$, we see that $p^{\ast} \alpha = {}^c \text{Ric}(\omega_{\text{RFK}})=0$. Hence, ${}^t \text{Ric}_{\omega}^{(1)}= \frac{t(1-n)-1}{2} d(J \vartheta)$, and for $\lambda>-1$, we set $\omega_{\lambda} : = \omega - \lambda d(J \vartheta)$. We observe that $\omega_{\lambda}^n = (1+\lambda)^{n-1} \omega^n$, and hence, \begin{eqnarray*}
{}^c \text{Ric}_{\omega_{\lambda}}^{(1)} &=& {}^c \text{Ric}_{\omega}^{(1)} \ = \ - \frac{n}{2} d(J \vartheta).
\end{eqnarray*}
Further, since $\partial \partial^{\ast} \omega_{\lambda} + \bar{\partial} \bar{\partial}^{\ast} \omega_{\lambda} = - \frac{n-1}{1+\lambda} d(J\vartheta)$, we see that \begin{eqnarray*}
    {}^t \text{Ric}_{\omega_{\lambda}}^{(1)} &=& {}^c \text{Ric}_{\omega_{\lambda}}^{(1)} + \frac{(t-1)}{2} (\partial \partial^{\ast} \omega_{\lambda} + \bar{\partial} \bar{\partial}^{\ast} \omega_{\lambda}) \\
    &=& - \frac{n}{2} d(J \vartheta) - \frac{(t-1)}{2} \frac{n-1}{1+\lambda} d(J\vartheta) \\
    &=& \left( \frac{-n(1+\lambda) +(1-t)(n-1)}{2(1+\lambda)} \right) d(J\vartheta).
\end{eqnarray*}
Setting $\lambda = \frac{t(1-n)-1}{n}$, we see that ${}^t \text{Ric}_{\omega_{\lambda}}^{(1)}=0$. Note that since $\lambda>-1$, we see that $t<1$.
\end{proof}

From the above theorem, we see that a large number of examples of first $t$--Gauduchon Ricci-flat metrics can be produced by producing examples of Sasaki--Einstein manifolds:

\begin{cor} 
There exist first $t$--Gauduchon Ricci-flat metrics for $t \in (-\infty,1)$ on the following classes of manifolds: \begin{itemize}
    \item[(i)] The suspension $\Sigma_{\Phi,\kappa}(\mathcal{Q})$, where $(\mathcal{Q},g_{\mathcal{Q}})$ is a compact Sasaki manifold of (real) dimension $3$.
    \item[(ii)] Compact Hermitian Weyl--Einstein manifolds. In particular, locally conformally hyperK\"ahler manifolds.
    \item[(iii)] Let $X:=\mathcal{L}(\textbf{a}) \times \mathbb{S}^1$, where $\mathcal{L}(\textbf{a}) := Y(\textbf{a})\cap \mathbb{S}^{2n+1}$ is the link of a Brieskorn--Pham singularity $$Y(\textbf{a}) \ : = \ \left( \sum_{k=0}^n z_k^{a_k} =0 \right) \subset \mathbb{C}^{n+1}, \hspace{1cm} n \geq 3, \hspace{1cm} a_0 \leq \cdots \leq a_n.$$ If $$1 \ < \ \sum_{k=0}^n \frac{1}{a_k} \ < \ 1 + \frac{n}{a_n}$$ then $X$ admits first $t$--Gauduchon Ricci-flat metrics for $t<1$.
    \item[(iv)] $\Lambda \times \mathbb{S}^1$, where $\Lambda$ is an odd-dimensional homotopy sphere that bounds a parallelizable manifold. 
\end{itemize}

In particular, all the above examples admit first Bismut Ricci-flat metrics,  first Hermitian conformal Ricci-flat metrics, first Minimal Ricci-flat metrics, and first Lichnerowicz Ricci-flat metrics. 
\end{cor}

\subsection*{3.4. The Gauduchon Scalar Curvatures}
In the following relations on the scalar curvature, we remind the reader that ${}^t \text{Scal}_{\omega} : = \text{tr}_{\omega} {}^t \text{Ric}_{\omega}^{(1)}$ is the familiar ($t$--Gauduchon) \textit{scalar curvature}, while ${}^t \widetilde{\text{Scal}}_{\omega} : = \text{tr}_{\omega} {}^t \text{Ric}_{\omega}^{(3)}$ is the ($t$--Gauduchon) \textit{altered scalar curvature}. Let $\tau$ be the (Chern) torsion $(1,0)$ form defined uniquely by $\partial \omega^{n-1} = \tau \wedge \omega^{n-1}$. It is well known that $\tau = T^k_{ik}e^i$ in any local frame. Taking the trace of the Ricci curvature relations in Corollary \ref{RicciRelationsCor}, we have: 

\begin{cor}\label{ScalarcurvatureRelations}
Let $(X, \omega)$ be a Hermitian manifold.  The $t$--Gauduchon scalar curvatures are given by \begin{eqnarray*}
{}^t \text{Scal}_{\omega} &=& t {}^c \text{Scal}_{\omega}  + (1-t) {}^c \widetilde{\text{Scal}}_{\omega} \\
{}^t \widetilde{\text{Scal}}_{\omega} &=& t {}^c \widetilde{\text{Scal}}_{\omega} + (1-t) {}^c \text{Scal}_{\omega} - \frac{(1-t)^2}{4} \left(|{}^cT|^2 + |\tau|^2\right).
\end{eqnarray*}
\end{cor}

The following is an immediate consequence of expressing the well-known Kobayashi--Wu \cite{KobayashiWu} vanishing theorem in terms of ${}^t \text{Scal}_{\omega}$:

\begin{prop}
Let $(X, \omega)$ be a compact Hermitian manifold. Suppose that one of the following conditions holds: \begin{itemize}
    \item[(i)] ${}^t \text{Scal}_{\omega} + (t-1) {}^c \widetilde{\text{Scal}}_{\omega} >0$ for some $t>0$. 
    \item[(ii)] ${}^t \text{Scal}_{\omega} + (t-1) {}^c \widetilde{\text{Scal}}_{\omega}<0$ for some $t<0$.
    \item[(iii)] ${}^t \text{Scal}_{\omega} + (1-t)(d^{\ast} \tau + | \tau |^2) \ > \ 0$ for some $t \in \mathbb{R}$.
\end{itemize}
Then the Kodaira dimension $\kappa(X)=-\infty$.
\end{prop} \begin{proof}
Statements (i) and (ii) imply by Corollary \ref{ScalarcurvatureRelations} that ${}^c \text{Scal}_{\omega}>0$. Hence, by the Kobayashi--Wu vanishing theorem \cite{KobayashiWu}, the Kodaira dimension is $\kappa(X)=-\infty$.  Statement (iii) is the same, after recalling that ${}^c \text{Scal}_{\omega} = {}^c \widetilde{\text{Scal}}_{\omega} + d^{\ast} \tau + | \tau |^2$, where $\tau$ denotes the Chern torsion $(1,0)$--form.
\end{proof}

\begin{cor}
Let $(X, \omega)$ be a Hermitian manifold. Denote by $\tau$ the (Chern) torsion $(1,0)$--form. Then \begin{eqnarray*}
{}^t \text{Scal}_{\omega} &=& {}^c \widetilde{\text{Scal}}_{\omega} + t d^{\ast} \tau + t | \tau |^2.
\end{eqnarray*}
In particular,  the total scalar curvature is given by \begin{eqnarray*}
\int_X {}^t \text{Scal}_{\omega} \omega^n &=& \int_X {}^c \widetilde{\text{Scal}}_{\omega} \omega^n + t \int_X | \tau |^2 \omega^n,
\end{eqnarray*}

and ${}^t \text{Scal}_{\omega} = {}^c \widetilde{\text{Scal}}_{\omega}$ if $t=0$ or the metric is balanced.
\end{cor}

The trace of \ref{LcKScalRelation} gives, for a locally conformally K\"ahler metric $\omega$, \begin{eqnarray*}
{}^t \text{Scal}_{\omega} &=& t {}^c \text{Scal}_{\omega} + (1-t) {}^c \widetilde{\text{Scal}}_{\omega}.
\end{eqnarray*}

In particular, \begin{itemize}
    \item[(i)] ${}^b \text{Scal}_{\omega} = - \text{Scal}_{\omega} + 2 {}^c \widetilde{\text{Scal}}_{\omega}$, 
    \item[(ii)] ${}^l \text{Scal}_{\omega} = {}^c \widetilde{\text{Scal}}_{\omega}$,
    \item[(iii)] ${}^t \text{Scal}_{\omega} = {}^s \text{Scal}_{\omega}$ if and only if $s=t$ or $\omega$ is balanced.
\end{itemize}

\begin{prop}
Let $(X, \omega)$ be a Hermitian manifold and $t \in \mathbb{R}$.  If $${}^t \text{Scal}_{\omega} \ = \ {}^t \widetilde{\text{Scal}}_{\omega}$$ and $t = \frac{1}{2}$, or $X$ is compact with $t \in (-\infty, -3-2\sqrt{3}] \cup [-3+2\sqrt{3},1) \cup (1, + \infty)$, then the metric is K\"ahler. In particular, if the Hermitian conformal scalar curvatures coincide, the metric is K\"ahler. 
\end{prop}

\begin{proof}
From Corollary \ref{ScalarcurvatureRelations}, we can write ${}^t \text{Scal}_{\omega} = {}^t \widetilde{\text{Scal}}_{\omega}$ as \begin{eqnarray*}
t {}^c \text{Scal}_{\omega} + (1-t) {}^c \widetilde{\text{Scal}}_{\omega} &=& t {}^c \widetilde{\text{Scal}}_{\omega} + (1-t) {}^c \text{Scal}_{\omega} - \frac{(1-t)^2}{4} \left( | {}^c T |^2 + | \tau |^2 \right).
\end{eqnarray*}
Equivalently, this reads \begin{eqnarray*}
(2t-1) ({}^c \text{Scal}_{\omega} - {}^c \widetilde{\text{Scal}}_{\omega}) + \frac{(1-t)^2}{4} (| {}^c T |^2 + |\tau |^2 ) \ = \ 0.
\end{eqnarray*}
If $t = \frac{1}{2}$, then $\frac{1}{8}(|{}^c T |^2 + | \tau |^2)=0$ implies that ${}^c T =0$, and hence the metric is K\"ahler. On the other hand, we know that ${}^c \text{Scal}_{\omega} - {}^c \widetilde{\text{Scal}}_{\omega} = d^{\ast} \tau + | \tau |^2$. Therefore, we can write \begin{eqnarray*}
(2t-1) ({}^c \text{Scal}_{\omega} - {}^c \widetilde{\text{Scal}}_{\omega}) + \frac{(1-t)^2}{4} (| {}^c T |^2 + |\tau |^2 ) &=& (2t-1)(d^{\ast} \tau + | \tau |^2) + \frac{(1-t)^2}{4} ( | {}^c T |^2 + | \tau |^2).
\end{eqnarray*}
For $X$ compact, integrating then yields \begin{eqnarray*}
(t^2+6t-3) \int_X | \tau |^2 \omega^n + (t-1)^2 \int_X | {}^c T |^2 \omega^n &=&0.
\end{eqnarray*}
If $t = 1$, then $4 \int_X | \tau |^2 \omega^n =0$ and the metric is balanced. If $t \in (-\infty, -3-2\sqrt{3}] \cup [-3+2\sqrt{3},1) \cup (1, + \infty)$,  then the metric is K\"ahler.
\end{proof}

\begin{rmk}
Note that $t=-1$ is not contained in the region $(-\infty, -3-2\sqrt{3}] \cup [-3+2\sqrt{3},1) \cup (1,+\infty)$. This is expected since there are plenty of examples of Bismut-flat Hermitian manifolds \cite{ChenZheng}. 
\end{rmk}

From \cite[Equation (5.5)]{FuZhou}, however, we can say the following:

\begin{prop}
Let $(X, \omega)$ be a compact Hermitian manifold. If the Bismut total scalar curvatures coincide $$\int_X {}^b \text{Scal}_{\omega} \omega^n = \int_X {}^b \widetilde{\text{Scal}}_{\omega} \omega^n,$$ then $\omega$ is balanced if and only if $\omega$ is K\"ahler.
\end{prop} \begin{proof}
From \cite[Equation (5.5)]{FuZhou}, we have \begin{eqnarray*}
{}^b \text{Scal}_{\omega} - {}^b \widetilde{\text{Scal}}_{\omega} &=& | d\omega |^2 - | \vartheta |^2 - \frac{3}{2} d^{\ast} \vartheta,
\end{eqnarray*}

where $\vartheta : = \frac{1}{n-1}(\tau + \bar{\tau})$ denotes the Lee form. In particular, if $X$ is compact and ${}^b \text{Scal}_{\omega} = {}^b \widetilde{\text{Scal}}_{\omega}$, then $\int_X ( | d \omega |^2 - | \vartheta |^2) \omega^n =0$. If $\vartheta=0$, then the metric balanced, and hence $\int_X | d\omega |^2 \omega^n=0$, i.e., $\omega$ is K\"ahler.
\end{proof}

\subsection*{3.5. The Gauduchon Holomorphic Bisectional Curvature}

The holomorphic bisectional curvature was introduced by Goldberg--Kobayashi \cite{GoldbergKobayashi}, and is the complex geometric analog of the sectional curvature in Riemannian geometry.  An altered variant of the (Chern) holomorphic bisectional curvature was introduced in \cite{BroderTangAltered}. We extend our understanding of these curvatures to the case of all Gauduchon connections:

\begin{defn}
Let $(X, \omega)$ be a Hermitian manifold.  The \textit{Gauduchon holomorphic bisectional curvature} is defined by $${}^t \text{HBC}_{\omega}(u,v) \ : = \ \frac{{}^t R(u, \overline{u}, v, \overline{v})}{| u |^2_{\omega} | v |_{\omega}^2},$$ where $u,v \in T^{1,0}X$.
\end{defn}

Before stating our main results concerning the Gauduchon holomorphic bisectional curvature, let us introduce the following:

\begin{defn}
Let $(X, \omega)$ be a Hermitian manifold. We define the \textit{Gauduchon altered holomorphic bisectional curvature} $${}^t \widetilde{\text{HBC}}_{\omega}(u,v) \ : = \ \frac{1}{| u |_{\omega}^2 | v |_{\omega}^2} {}^t R(u, \overline{v}, v, \overline{u}),$$ where $u,v \in T^{1,0}X$. 
\end{defn}

\begin{rmk}\label{AlteredRmk}
In \cite{BroderTangAltered},  the first named author, together with Kai Tang, used the notation $\widetilde{\text{HBC}}_{\omega}$ to refer to ${}^c R(u, \overline{u}, v, \overline{v}) + {}^c R(v, \overline{v}, u, \overline{u})$, which was called the altered holomorphic bisectional curvature.  We believe that both the choice of notation and terminology for  ${}^c R(u, \overline{u}, v, \overline{v}) + {}^c R(v, \overline{v}, u, \overline{u})$ should be abandoned. Indeed, we contend that the term \textit{altered} should refer only to variants of the curvature which have entries repeated entries occurring in the second+third entries, or first + fourth entries. For instance, the third and fourth Ricci curvatures: $\text{Ric}^{(3)}(\cdot, \cdot) = \sum_k R(\cdot, \bar{e}_k, e_k, \cdot)$ and $\text{Ric}^{(4)}(\cdot, \cdot) = \sum_k R(e_k, \cdot, \cdot, \bar{e}_k)$ would be considered `altered Ricci curvatures'. The trace of these `altered Ricci curvatures' yields the `altered scalar curvatures'.  Another instance is given by the `altered real bisectional curvature' \cite{BroderTangAltered}, which is given by ${}^t \widetilde{\text{RBC}}_{\omega}(v) = \sum_{\alpha, \gamma} {}^t R_{\alpha \bar{\gamma} \gamma \bar{\alpha}} v_{\alpha} v_{\gamma}$, where $v \in \mathbb{R}^n \backslash \{ 0 \}$ has unit length.  We will break this convention only in one instance (due to lack of a better name) -- the `altered holomorphic sectional curvature' $\widetilde{\text{HSC}}$, which we will discuss in great detail in the next section.
\end{rmk}

 \begin{prop}
Let $(X, \omega)$ be a compact Hermitian manifold of (complex) dimension $n$.  Then \begin{eqnarray*}
\frac{2\pi}{n | v |^2} {}^t \text{Ric}_{\omega}^{(1)} (v,\overline{v}) &=& \dashint_{\mathbb{S}^{2n-1}} {}^t \text{HBC}_{\omega}([v], [w]) d \sigma(w), \\
\frac{2\pi}{n | v |^2} {}^t \text{Ric}_{\omega}^{(2)} (w,\overline{w}) &=& \dashint_{\mathbb{S}^{2n-1}} {}^t \text{HBC}_{\omega}([v], [w]) d \sigma(v),
\end{eqnarray*}
\end{prop}

where $d\sigma$ is the Lebesgue measure on $\mathbb{S}^{2n-1} \subset T_x^{1,0}X$ for each $x \in X$ and $\dashint : = \frac{(n-1)!}{2\pi^n} \int_{\mathbb{S}^{2n-1}}$.  \begin{proof}
Fix a point $x \in X$ and write ${}^t R_{i \overline{j} k \overline{\ell}}$ for the components of the $(1,1)$--part of the $t$--Gauduchon curvature tensor in a local frame near $x \in X$.  Then \begin{eqnarray*}
\dashint_{\mathbb{S}^{2n-1}} {}^t \text{HBC}_{\omega}([v], [w]) d \sigma(w) &=& \frac{1}{| v |^2} \sum_{i,j,k, \ell=1}^n {}^t R_{i \overline{j} k \overline{\ell}} v_i \overline{v}_j \dashint_{\mathbb{S}^{2n-1}} w_k \overline{w}_{\ell} d\sigma(w) \\
&=& \frac{1}{n | v |^2} \sum_{i,j,k,\ell=1}^n {}^t R_{i \overline{j} k \overline{\ell}} v_i \overline{v}_j\delta_k^{\ell} \\
&=& \frac{1}{n | v |^2} \sum_{i,j,k=1}^n {}^t R_{i \overline{j} k \overline{k}} v_i \overline{v}_j \ = \ \frac{2\pi}{n | v |^2} {}^t \text{Ric}_{\omega}^{(1)} (v,\overline{v}).
\end{eqnarray*}
The same argument with $d\sigma(w)$ replaced by $d\sigma(v)$ shows that ${}^t \text{HBC}_{\omega}$ dominates ${}^t \text{Ric}_{\omega}^{(2)}$.
\end{proof}

\begin{cor}
Let $(X, \omega)$ be a compact Hermitian manifold. The $t$--Gauduchon bisectional curvature dominates the first and second $t$--Gauduchon Ricci curvatures. In particular, if ${}^t \text{HBC}_{\omega}>0$ or ${}^t \text{HBC}_{\omega}<0$, then $X$ supports a pluriclosed metric.
\end{cor}

The Berger argument does not imply that ${}^t \text{HBC}_{\omega}$ dominates ${}^t \text{Ric}_{\omega}^{(3)}$ or ${}^t \text{Ric}_{\omega}^{(4)}$, hence it is natural to ask the following:

\begin{q}
Let $(X, \omega)$ be a compact Hermitian manifold.  Does the $t$--Gauduchon bisectional curvature dominate the third and fourth $t$--Gauduchon Ricci curvatures?
\end{q}

On the other hand, we have the analogous result for the $t$--Gauduchon altered bisectional curvature:

\begin{prop}
Let $(X, \omega)$ be a compact Hermitian manifold of (complex) dimension $n$.  Then \begin{eqnarray*}
\frac{2\pi}{n | v |^2} {}^t \text{Ric}_{\omega}^{(3)} (v,\overline{v}) &=& \dashint_{\mathbb{S}^{2n-1}} {}^t \widetilde{\text{HBC}}_{\omega}([v], [w]) d \sigma(w), \\
\frac{2\pi}{n | v |^2} {}^t \text{Ric}_{\omega}^{(4)} (w,\overline{w}) &=& \dashint_{\mathbb{S}^{2n-1}} {}^t \widetilde{\text{HBC}}_{\omega}([v], [w]) d \sigma(v),
\end{eqnarray*}
where $d\sigma$ is the Lebesgue measure on $\mathbb{S}^{2n-1} \subset T_x^{1,0}X$ for each $x \in X$.
\end{prop} \begin{proof}
In a similar manner to the proof of the previous result, we see that \begin{eqnarray*}
\dashint_{\mathbb{S}^{2n-1}} {}^t \widetilde{\text{HBC}}_{\omega}([v],[w]) d\sigma(w) &=& \frac{1}{| v |^2} \sum_{i,j,k,\ell=1}^n {}^t R_{i \overline{j} k \overline{\ell}} v_i \overline{v}_{\ell} \dashint_{\mathbb{S}^{2n-1}} w_j \overline{w}_k d\sigma(w) \\
&=& \frac{2\pi}{n| v |^2} \sum_{i,j,k,\ell=1}^n {}^t R_{i \overline{k} k \overline{\ell}} v_i \overline{v}_{\ell} \ = \ \frac{2\pi}{n | v |^2} {}^t \text{Ric}_{\omega}^{(3)}(v, \overline{v}).
\end{eqnarray*}

With $d\sigma(w)$ replaced by $d\sigma(v)$, we see that ${}^t \widetilde{\text{HBC}}_{\omega}$ dominates ${}^t \text{Ric}_{\omega}^{(4)}$.
\end{proof}

\section{The Gauduchon holomorphic Sectional Curvature}
Let $(X,\omega)$ be a Hermitian manifold. Recall that the $t$--Gauduchon holomorphic sectional curvature is the map ${}^t \operatorname{HSC}_\omega \colon T^{1,0}X \to \mathbb{R}$ given by
\[
{}^t\operatorname{HSC}_\omega(v) \:=\ \frac{1}{|v|^4}{}^tR(v,\overline v,v,\overline v) \ =\ {}^t\operatorname{HBC}_\omega(v,v) \ =\ \widetilde{{}^t\operatorname{HBC}_\omega}(v,v),
\]
for all $v \in T^{1,0}X$. We begin this section by considering when two different Gauduchon holomorphic sectional curvatures can agree and prove Theorem \ref{MainThmHSC2}:
\begin{thm}
Let $(X,\omega)$ be a Hermitian manifold, if ${}^t\operatorname{HSC} \equiv {}^s\operatorname{HSC}$, then either $t = s$, $t = 2-s$, or $\omega$ is K\"ahler.
\end{thm}
\begin{proof}
Let $v \in \mathbb{S}(T^{1,0}X)$ and let $\{e_1 = v, e_2 ,\dots,e_n\}$ be a unitary frame. In this frame, ${}^tR_{1\overline 1 1 \overline 1} = {}^t\operatorname{HSC}(v)$. Then, by Equation (\ref{eqn:tCurv11}),
\[
{}^tR_{1\overline 1 1\overline 1}+\left(\frac{t-1}{2}\right)^2\sum_r |T_{1r}^1|^2 = {}^cR_{1\overline 1 1 \overline 1} = {}^sR_{1\overline 1 1\overline 1}+\left(\frac{s-1}{2}\right)^2\sum_r |T_{1r}^1|^2.
\]
Since ${}^t\operatorname{HSC} \equiv {}^s\operatorname{HSC}$, we yield
$0 = \left((1-t)^2 - (1-s)^2\right)\sum_r|T^1_{1r}|^2$, from which it follows that either $t = s$, $t = 2-s$ or $\langle T(v,w), \overline{v} \rangle = 0$ for all $v,w \in T^{1,0}X$. This implies that $T \equiv 0$. Indeed, for each $w \in T^{1,0}X$, the endomorphism $T_w := T(w,\cdot)$ is skew-Hermitian, hence diagonalisable over $\mathbb{C}$, since for all $v \in T^{1,0}X$, $0 = \langle (T_w + T^*_w)v,\overline {v} \rangle$. Moreover, the condition $T_w v \perp v$ implies the eigenvalues of $T_w$ are identically zero. It follows that $T_w \equiv 0$ as required.
\end{proof}

\begin{rmk}
This phenomenon appears implicitly in the recent work of Chen--Nie \cite{ChenNieHSC} for compact Hermitian surfaces. Indeed,  they show that if $(X, \omega)$ is a compact Hermitian surface with ${}^t \text{HSC}_{\omega} \equiv c$, then for $t \in \mathbb{R} \backslash \{ -1, 3 \}$,  the metric $\omega$ is K\"ahler. On the other hand, for $t=-1$ or $t=3$,  $(X, \omega)$ must be an isosceles Hopf surface. Note that the map $t \mapsto 2-t$ interchanges $t=-1$ and $t=3$,  and fixes $t=1$. 
\end{rmk}

Denote by $\mathcal{F}_X \to X$ the unitary frame bundle. From considerations of the Schwarz lemma in the Hermitian category, Yang--Zheng \cite{YangZhengRBC} (see also \cite{LeeStreets}) introduced the following curvature:

\subsection*{4.1. The Gauduchon Real Bisectional Curvature}
\begin{defn}
Let $(X,\omega)$ be a Hermitian manifold. The \textit{$t$--Gauduchon real bisectional curvature} ${}^t \text{RBC}_{\omega}$ is the \index{real bisectional curvature} function $${}^t \text{RBC}_{\omega} : \mathcal{F}_X \times \mathbb{R}^n \backslash \{ 0 \} \longrightarrow \mathbb{R}, \hspace*{1cm} {}^t \text{RBC}_{\omega}(v) \ := \ \frac{1}{| v |^2} \sum_{\alpha, \gamma} {}^t R_{\alpha \overline{\alpha} \gamma \overline{\gamma}} v_{\alpha} v_{\gamma}.$$ Here, ${}^t R_{\alpha \overline{\beta} \gamma \overline{\delta}}$ denote the components of the $t$--Gauduchon curvature tensor with respect to the local unitary frame, and $v = (v_1, ..., v_n) \in \mathbb{R}^n \backslash \{ 0 \}$. We say that ${}^t \text{RBC}_{\omega} \leq \kappa$ if $\max_{(e,\lambda) \in \mathcal{F}_X \times \mathbb{R}^n \backslash \{ 0 \} } {}^t \text{RBC}_{\omega}(e,\lambda) \leq \kappa$. Similar definitions apply for ${}^t \text{RBC}_{\omega} \geq \kappa$ and ${}^t \text{RBC}_{\omega} \equiv \kappa$.
\end{defn}

\begin{rmk}
If we let $v \in T^{1,0}X$ and choose a local unitary frame $\{ e_{\alpha} \}$ for $T^{1,0}X$ such that $v$ is parallel to $e_1$,  then with respect to this frame, we have ${}^t \text{RBC}_{\omega}(v) = {}^t R_{1 \overline{1}1 \overline{1}}  = {}^t \text{HSC}_{\omega}(v)$. Hence, the ($t$--Gauduchon) real bisectional curvature dominates the ($t$--Gauduchon) holomorphic sectional curvature.  The (Chern) real bisectional curvature is not strong enough to dominate the (Chern) Ricci curvatures. A local example was constructed in \cite{YangZhengRBC}. 
\end{rmk}

\begin{rmk}
The (Chern) real bisectional curvature was introduced by Yang--Zheng \cite{YangZhengRBC}. The motivation comes exclusively from the Schwarz lemma in the Hermitian category. Indeed, if $f : (X, \omega_g) \to (Y, \omega_h)$ is a holomorphic map between Hermitian manifolds, the Chern--Lu incarnation of the Schwarz lemma (\cite{Chern, Lu, Royden, YauSchwarz, YangZhengRBC, Rubinstein, BroderSBC1, BroderSBC2}) requires a lower bound on the second Chern Ricci curvature of $\omega_g$ and if $\omega_h$ is K\"ahler (or more generally, (Chern) K\"ahler-like in the sense \cite{YangZhengCurvature}), Royden's polarization argument \cite{Royden} shows that the target curvature term is controlled by the (Chern) holomorphic sectional curvature. For a general Hermitian metric, however, the target curvature term is not controlled by the (Chern) holomorphic sectional curvature, but rather the (Chern) real bisectional curvature. A more refined variant -- the (Chern) \textit{second Schwarz bisectional curvature} -- was introduced by the first named author in \cite{BroderSBC1, BroderSBC2}, together with its links to convex geometry. This initiated a program to study various curvatures by viewing them as quadratic form-valued functions on the unitary frame bundle, which has been carried out in \cite{BroderSBC1, BroderSBC2, BroderGraph, BroderLA, BroderQOBC, BroderTangAltered}.
\end{rmk}

\subsection*{4.2. The Gauduchon Altered Real Bisectional Curvature}
In \cite{BroderTangAltered}, the first named author, together with Kai Tang, introduced the following variant of the real bisectional curvature:

\begin{defn}
Let $(X,\omega)$ be a Hermitian manifold. The \textit{Gauduchon altered real bisectional curvature} ${}^t \widetilde{\text{RBC}}_{\omega}$ is the function $${}^t \widetilde{\text{RBC}}_{\omega} : \mathcal{F}_X \times \mathbb{R}^n \backslash \{ 0 \} \longrightarrow \mathbb{R}, \hspace*{1cm} {}^t \widetilde{\text{RBC}}_{\omega}(v) \ := \ \frac{1}{| v |^2} \sum_{\alpha, \gamma} {}^t R_{\alpha \overline{\gamma} \gamma \overline{\alpha}} v_{\alpha} v_{\gamma}.$$
\end{defn}

Analogous to the $t$--Gauduchon real bisectional curvature, the $t$--Gauduchon altered real bisectional curvature dominates the $t$--Gauduchon holomorphic sectional curvature.

\subsection*{4.3. The Gauduchon Altered Holomorphic Sectional Curvature}
The altered real bisectional curvature exhibits similar behavior to the real bisectional curvature. We invite the reader to consult \cite{BroderTangAltered} for the details. The primary motivation for the study of the altered real bisectional curvature is given by the fact that the sum of the real bisectional curvature and the altered real bisectional curvature yields a curvature that is comparable to the holomorphic sectional curvature:

\begin{defn}
Let $(X, \omega)$ be a Hermitian manifold.  The \textit{Gauduchon altered holomorphic sectional curvature} is given in any local unitary frame by \begin{eqnarray*}
{}^t \widetilde{\text{HSC}}_{\omega}(v) & :=  & \frac{1}{| v |_{\omega}^2} \sum_{\alpha, \gamma} \left( {}^t R_{\alpha \bar{\alpha} \gamma \bar{\gamma}} + {}^t R_{\alpha \bar{\gamma} \gamma \bar{\alpha}} \right) v_{\alpha} v_{\gamma},
\end{eqnarray*}
where $v = (v_1, ..., v_n) \in \mathbb{R}^n \backslash \{ 0 \}$. 
\end{defn}

\begin{rmk}
The (Chern) altered holomorphic sectional curvature was formally introduced by the first named author and Kai Tang in \cite{BroderTangAltered}, although it appeared implicitly in earlier works (see, e.g., \cite{YangZhengRBC}). 
\end{rmk}

\begin{rmk}
Let $(X, \omega)$ be a Hermitian manifold. The $t$--Gauduchon altered holomorphic sectional curvature ${}^t \widetilde{\text{HSC}}_{\omega}$ and $t$--Gauduchon holomorphic sectional curvature ${}^t \text{HSC}_{\omega}$ are comparable in the sense that they have the same sign.

\end{rmk}

\begin{prop}
Let $(X, \omega)$ be a Hermitian manifold.  The Gauduchon real bisectional curvature and altered real bisectional curvatures are given in any unitary frame by  \begin{eqnarray*}
{}^t \text{RBC}_{\omega}(v) &=& t {}^c \text{RBC}_{\omega}(v) + (1-t) {}^c \widetilde{\text{RBC}}_{\omega}(v) \\
&& \hspace*{4cm} + \frac{(t-1)^2}{4 | v |_{\omega}^2} \sum_{i,k,q} \left( {}^c T_{ik}^q \overline{{}^c T_{ik}^q} - {}^c T_{i q}^k \overline{{}^c T_{iq}^k} \right) v_i v_k,\\ 
{}^t \widetilde{\text{RBC}}_{\omega}(v) &=& t {}^c \widetilde{\text{RBC}}_{\omega}(v) + (1-t) {}^c \text{RBC}_{\omega}(v) \\
&& \hspace*{4cm} + \frac{(t-1)^2}{4 | v |_{\omega}^2} \sum_{i,k,q} \left( {}^c T_{ik}^q \overline{{}^c T_{ki}^q} - {}^c T_{iq}^i \overline{{}^c T_{kq}^k} \right) v_i v_k. 
\end{eqnarray*} \end{prop} \begin{proof}
In any unitary frame, the components of the Gauduchon curvature tensor are given by \begin{eqnarray*}
{}^t R_{i \overline{j} k \overline{\ell}} &=& {}^c R_{i \overline{j} k \overline{\ell}} + \frac{(1-t)}{2} \left( {}^c R_{k \overline{j} i \overline{\ell}} - 2 {}^c R_{i \overline{j} k \overline{\ell}} + {}^c R_{i \overline{\ell} k \overline{j}}  \right) \\
&& \hspace*{4cm} + \frac{(t-1)^2}{4} \sum_{q} \left( {}^c T_{i k}^q \overline{{}^c T_{j\ell}^q} - {}^c T_{i q}^{\ell} \overline{{}^c T_{jq}^{k}}\right).
\end{eqnarray*}
Hence,  \begin{eqnarray*}
{}^t R_{i \overline{i} k \overline{k}} &=& {}^c R_{i \overline{i} k \overline{k}} + \frac{(1-t)}{2} \left( {}^c R_{k \overline{i} i \overline{k}} - 2 {}^c R_{i \overline{i} k \overline{k}} + {}^c R_{i \overline{k} k \overline{i}} \right) \\
&& \hspace*{4cm} + \frac{(t-1)^2}{4} \sum_{q} \left( {}^c T_{ik}^q \overline{T_{ik}^q} - {}^c T_{i q}^k \overline{{}^c T_{iq}^k} \right),  \\
{}^t R_{i \overline{k} k \overline{i}} &=& {}^c R_{i \overline{k} k \overline{i}} + \frac{(1-t)}{2} \left( {}^c R_{k \overline{k} i \overline{i}} - 2 {}^c R_{i \overline{k} k \overline{i}} + {}^c R_{i \overline{i} k \overline{k}} \right) \\
&& \hspace*{4cm} + \frac{(t-1)^2}{4} \sum_{q} \left( {}^c T_{ik}^q \overline{T_{ik}^q} - {}^c T_{iq}^i \overline{{}^c T_{kq}^i} \right).
\end{eqnarray*}
\end{proof}

Since the altered holomorphic sectional curvature is defined to be the sum of the real bisectional and altered real bisectional curvature,  we immediately deduce:

\begin{thmfix}
Let $(X, \omega)$ be a Hermitian manifold.  The Gauduchon altered holomorphic sectional curvature is given by \begin{eqnarray}
{}^t \widetilde{\text{HSC}}_{\omega}(v) &=& {}^c \widetilde{\text{HSC}}_{\omega}(v) - \frac{(t-1)^2}{4 | v |_{\omega}^2} \sum_{i,k,q} \left( {}^c T_{iq}^i \overline{{}^c T_{kq}^k} + {}^c T_{iq}^k \overline{{}^c T_{iq}^k}\right)v_i v_k.
\end{eqnarray}

In particular, ${}^t \widetilde{\text{HSC}}_{\omega}  \leq {}^c \widetilde{\text{HSC}}_{\omega}$ for all $t \in \mathbb{R}$ and equality holds if and only if $t=1$ or $\omega$ is K\"ahler.  
\end{thmfix}

\begin{rmk}
Note that we can already see the monotonicity theorem from Theorem \ref{MainCurvatureFormula}. Indeed, if we consider a unitary frame $\{ e_{\alpha} \}$ such that $v \in T^{1,0}X$ is parallel to $e_1$, then the formula in Theorem \ref{MainCurvatureFormula} implies that \begin{eqnarray*}
{}^t R_{1 \overline{1} 1 \overline{1}} &=& {}^c R_{1 \overline{1} 1 \overline{1}} - \frac{(1-t)^2}{4} \sum_r | {}^c T_{1r}^1 |^2.
\end{eqnarray*}
\end{rmk}

\begin{cor}
Let $(X, \omega)$ be a Hermitian manifold.  If ${}^t \widetilde{\text{HSC}}_{\omega} \equiv {}^s \widetilde{\text{HSC}}_{\omega}$, then $t=s$, $t=2-s$, or $\omega$ is K\"ahler. 
\end{cor}

Since the altered holomorphic sectional curvature is comparable to the holomorphic sectional curvature,  the following useful consequences of the above monotonicity result are easily obtained: 

\begin{cor}
Let $(X, \omega)$ be a Hermitian manifold.  \begin{itemize}
\item[(i)] If ${}^c \text{HSC}_{\omega} \leq 0$, then ${}^t \text{HSC}_{\omega} \leq 0$ for all $t \in \mathbb{R}$.
\item[(ii)] If ${}^t \text{HSC}_{\omega}>0$ for some $t \in \mathbb{R}$,  then ${}^c \text{HSC}_{\omega}>0$.
\end{itemize}
\end{cor}

\begin{rmk}\label{MonotonicityRmk}
Recall that in \cite{Demailly1997}, Demailly constructed a compact (projective) Kobayashi hyperbolic surface which does not admit a Hermitian metric of negative Chern holomorphic sectional curvature. In light of the above corollary, we see that negative Chern holomorphic sectional curvature is the most restrictive among Gauduchon connections.  It is, therefore, natural to raise the following question:
\end{rmk}

\begin{q}
Let $X$ be a compact Kobayashi hyperbolic manifold.  Does $X$ admit a Hermitian metric with ${}^t \text{HSC}_{\omega} <0$ for some (range of) $t \in \mathbb{R}$? \end{q}

\begin{rmk}
The argument used to prove Demailly's algebraic hyperbolicity criterion \cite{Demailly1997} does not readily extend to Gauduchon connections. Indeed, the argument hinges upon the subbundle decreasing property for the (Chern) holomorphic sectional curvature. In particular, one needs to be able to compare the connection on $T^{1,0}X$ with the connection on $T^{1,0}\mathcal{C} \otimes \mathcal{O}_{\mathcal{C}}(\mathcal{D})$. Here, $f : \mathcal{C} \to X$ is a holomorphic map from a compact Riemann surface $\mathcal{C}$ with discriminant locus $\mathcal{D}$. For the Chern connection, it is clear, but defining a Gauduchon connection on $T^{1,0}\mathcal{C} \otimes \mathcal{O}_{\mathcal{C}}(\mathcal{D})$ requires a choice. 
\end{rmk}

In the other direction, Yang \cite{YangHSCYau} showed that a compact K\"ahler manifold with $\text{HSC}_{\omega}>0$ is projective and rationally connected. The standard metric on the Hopf surface $\mathbb{S}^3 \times \mathbb{S}^1$ has ${}^c \text{HSC}_{\omega}>0$, but the Hopf surface supports no rational curves. Indeed, since $\mathbb{P}^1$ is simply connected, any holomorphic map $\mathbb{P}^1 \to \mathbb{S}^3 \times \mathbb{S}^1$ lifts to the universal cover $\mathbb{P}^1 \to \mathbb{C}^2 \backslash \{ 0 \} \subset \mathbb{C}^2$, but such holomorphic maps are necessarily constant. Since the positivity of the Chern holomorphic sectional curvature ${}^c \text{HSC}_{\omega}$ is the weakest constraint on the Gauduchon holomorphic sectional curvatures, it is natural to ask the following:

\begin{q}
Let $(X, \omega)$ be a compact Hermitian manifold. If ${}^t \text{HSC}_{\omega}>0$ for some (range of) $t \in \mathbb{R}$, is $X$ rationally connected? 
\end{q}

We suspect that the Gauduchon holomorphic sectional curvature may play a role in theory of special manifolds \cite{CampanaSpecial} and Oka manifolds \cite{Oka} (c.f., Remark \ref{MonotonicityRmk}). Let $X$ be a compact complex manifold. Recall that a saturated coherent sheaf $\mathcal{L} \subset \Omega_X^p$ of rank one is said to be a \textit{Bogomolov sheaf} if it has Kodaira--Iitaka dimension $\kappa(X,\mathcal{L})=p>0$.  Morally, Bogomolov sheaves correspond to the canonical bundles of varieties of general type dominated by $X$.  More precisely,  a Bogomolov sheaf comes from a dominant orbifold morphism to an orbifold pair of general type (see \cite{CampanaSpecial, CampanaWinkelmann} for details and precise statements).  A complex manifold is said to be \textit{special} if there are no Bogomolov sheaves on $X$. It is known that rationally connected manifolds are special.  Recall that a complex manifold $X$ is said to be \textit{Oka} if every holomorphic map from a neighborhood of a compact convex set $K \subset \mathbb{C}^n$ to $X$ is a uniform limit on $K$ of entire maps $\mathbb{C}^n \to X$.  By \cite{CampanaWinkelmann}, a projective Oka manifold is special, but in general, the relationship between special and Oka manifolds remains a mystery. We pose the following question:

\begin{q}
Let $(X,\omega)$ be a (compact) complex manifold.  Can specialness or the Oka property be characterized by ${}^t \text{HSC}_{\omega}>0$ for some (range of) $t \in \mathbb{R}$? 
\end{q}

We invite the reader to consult \cite[$\S 11$]{OkaSurvey} for the present state of affairs concerning a curvature characterization of Oka manifolds.


\begin{thebibliography}{9}

\bibitem{ADM}
Apostolov, V., Davidov, J., Muskarov, O., \textit{Compact self-dual Hermitian surfaces}, Trans. Amer. Math. Soc., \textbf{348} 1996, pp. 3051--3063


\bibitem{BalasGauduchon}
\underline{\phantom{Balas, A.}}, Gauduchon, P., {Any Hermitian metric of constant nonpositive (Hermitian) holomorphic sectional curvature on a compact complex surface is K\"ahler}, Math. Z. 1985, 190, 39--43

\bibitem{Barbaro}
Barbaro, G.,  On the curvature of the Bismut connection: Bismut Yamabe problem and Calabi-Yau with torsion metrics, arXiv:2109.06159


\bibitem{Bismut}
Bismut, J.-M., A local index theorem for non-K\"ahler manifolds, Math. Ann., \textbf{284}, 4 (1989), pp. 681--699.

\bibitem{Blair}
Blair, D., Riemannian geometry of contact and symplectic manifolds. Second edition. Progress in Mathematics, 203, Birkh\"auser Boston, Inc., Boston, MA, 2010.


\bibitem{BoyerGalicki1999}
Boyer, C., Galicki, K., $3$-Sasakian manifolds, Surveys in differential geometry: essays on Einstein manifolds, 123--184. Surv. Differ. Geom., VI, Int. Press, Boston (1999).

\bibitem{BoG1}
Boyer, C., Galicki, K., Sasakian geometry, Oxford Mathematical Monographs. Oxford University Press, Oxford, 2008.

\bibitem{BroderSBC1}
Broder, K., {The Schwarz Lemma in K\"ahler and Non-K\"ahler Geometry}, arXiv:2109.06331

\bibitem{BroderSBC2}
Broder, K., {The Schwarz Lemma: An Odyssey}, to appear in the Rocky Mountain J. of Mat., arXiv:2110.04989

\bibitem{BroderGraph}
Broder, K., On the nonnegativity of the Dirichlet energy of a weighted graph, Bull.  of the Aust. Math. Soc., 1-5. doi:10.1017/S0004972721001015

\bibitem{BroderLA}
Broder, K., An eigenvalue characterization of the dual EDM cone, Bulletin of the Australian Mathematical Society, 1-3. doi:10.1017/S0004972721000915

\bibitem{BroderQOBC}
Broder, K., Remarks on the Quadratic Orthogonal Bisectional Curvature,  J. Geom. 113, 39 (2022). \url{https://doi.org/10.1007/s00022-022-00653-3}

\bibitem{BroderThesis}
Broder, K., Complex manifolds of hyperbolic and non-hyperbolic-type, Ph.D. thesis, available on the author's webpage: \url{
https://github.com/Kylebroder/PhDThesis/raw/main/PHD%20Book.pdf 
}


\bibitem{BroderTangAltered}
Broder, K., Tang, K., On the altered holomorphic curvatures of Hermitian manifolds, arXiv:2201.03666



\bibitem{CalderbankPedersen}
Calderbank,  D., Pedersen,  H.,  Einstein-Weyl geometry, in ``Surveys in Differential Geometry: Essays on Einstein manifolds", M. Wang, C. LeBrun eds., International Press 2000, pp. 387--423.


\bibitem{CampanaSpecial}
Campana, F., Orbifolds, special varieties and classification theory. Ann. Inst. Fourier (Grenoble) 54 (2004), no. 3, 499--630

\bibitem{CampanaWinkelmann}
Campana, F., Winkelmann, J., On the h-principle and specialness for complex projective manifolds. Algebr. Geom. 2 (2015), no. 3, 298–314


\bibitem{ChenChenNie}
Chen, H., Chen, L., Nie, X., {Chern-Ricci curvatures, holomorphic sectional curvature and Hermitian metrics}, Sci. China - Math, 2020, 63

\bibitem{ChenNieHSC}
Chen, H., Nie, X., Compact Hermitian surfaces with pointwise constant Gauduchon holomorphic sectional curvature, arXiv:2201.13083

\bibitem{ChenZheng}
Chen, S., Zheng, F., On Strominger space forms, J. Geom. Anal. 32 (2022), no. 4, Paper No. 141, 21 pp.

\bibitem{Chern}
Chern, S.-S., On holomorphic mappings of hermi- tian manifolds of the same dimension. In Entire Functions and Related Parts of Analysis (Proc. Sympos. Pure Math., La Jolla, Calif., 1966). Amer. Math. Soc., Providence, R.I. 1968.  pp. 157–170.

\bibitem{CheungPhD}
Cheung,  C.-K.,  Negative holomorphic sectional curvature and hyperbolic manifold, Ph.D. thesis, Berkeley, 1988

\bibitem{Chiose}
Chiose,  I., Obstructions to the existence of K\"ahler structures on compact complex manifolds. Proc. Amer. Math. Soc. 142 (2014), no. 10, 3561–3568.

\bibitem{Correa}
Correa,  E. M., Levi-Civita Ricci-flat metrics on non-K\"ahler Calabi-Yau manifolds, arXiv:2204.04824v3

\bibitem{Cowen}
Cowen, M. J.,  Families of negatively curved Hermitian manifolds,  Proceedings of the AM.S.  \textbf{39},  no. 2, pp. 362--366.


\bibitem{Demailly1997}
Demailly, J.-P., \textit{Algebraic criteria for Kobayashi hyperbolic projective varieties and jet differentials}, Algebraic geometry--Santa Cruz 1995, Proc. Sympos. Pure Math., vol. \textbf{62}, Amer. Math. Soc., Providence, RI, 1997, pp. 285--360. 

\bibitem{DiverioSurvey}
Diverio, S., \textit{Kobayashi hyperbolicity, negativity of the curvature and positivity of the canonical bundle}, arXiv:2011.11379


\bibitem{DIU}
Dragomir, S., Ichiyama, T., Urakawa, H., Yang--Mills theory and conjugate connections, Diff. Geom. Appl. \textbf{18}, no. 2, pp. 229--238, (2003).


\bibitem{Oka}
Forstneri$\check{\text{c}}$, F., Oka manifolds, C. R. Math. Acad. Sci. Paris 347 (2009), 1017–20.

\bibitem{OkaSurvey}
Forstneri$\check{\text{c}}$, F., Recent developments in Oka theory, arXiv:2006.07888v4

\bibitem{FuZhou}
Fu, J.-X., Zhou, X., Scalar curvatures in almost Hermitian geometry and some applications, arXiv:1901.10130


\bibitem{GauduchonEinsteinWeyl}
Gauduchon, P.,  Structures de Weyl-Einstein, espaces de twisteurs et vari\'et\'es de type $S^1 \times S^3$, J. reine Angew. Math. \textbf{469} (1995), pp. 1--50.



\bibitem{GauduchonHermitianConnections}
Gauduchon, P.,  Hermitian connections and Dirac operators,  Boll. Un. Mat. Ital. B (7) 11 (1997), no. 2, suppl., 257--288

\bibitem{GiniOrneaParton}
Gini, R., Ornea, L., Parton, M., Locally conformal K\"ahler reduction, J. Reine Angew. Math., 581: 1--21, 2005.

\bibitem{GoldbergKobayashi}
Goldberg, S., Kobayashi, S.,  Holomorphic bisectional curvature. J. Differential Geometry 1 (1967), 225--233. 

\bibitem{GrauertReckziegel}
Grauert, H.,  Reckziegel, H.,  Hermitesche metriken und nortnale familien holomorpher abbildungen,  Math. Z. \textbf{89} (1965) pp. 108--125.


\bibitem{GreeneWu}
Greene,  R. E., Wu, H.,  Function Theory on Manifolds Which Possess a Pole,  Lecture Notes in Mathematics,  699,  Springer-Verlag Berlin Heidelberg New York, 1979.


\bibitem{HeLiuYang}
He, J., Liu, K., Yang, X., Levi-Civita Ricci-flat metrics on compact complex manifolds,J. Geom. Anal. 30 (2020), no. 1, pp. 646-–666 


\bibitem{IvanovPapadopoulos}
Ivanov, S., Papadopoulos, G., Vanishing Theorems and String Backgrounds, Classical Quantum Gravity 18 (2001), no. 6, pp. 1089--1110. 

\bibitem{Kashiwada}
Kashiwada, T., A note on a Riemannian space with Sasakian $3$-structure, Nat. Sci. Rep. Ochanomizu Univ., 22 (1971), pp. 1--2.


\bibitem{KobayashiWu}
Kobayashi, S., Wu, H.-S., On holomorphic sections of certain Hermitian vector bundles, Math. Ann. \textbf{189} (1970), p. 1--4.


\bibitem{KoszulMalgrange}
Koszul, J.-L., Malgrange, B., Sur certaine structures fibr\'ees complexes, arch. mat, vol IX, 1958

\bibitem{LafuenteStanfield}
Lafuente, R., Stanfield, J.
\newblock {H}ermitian manifolds with flat {G}auduchon connections. arXiv:2204.08170

\bibitem{LeeStreets}
Lee, M.-C., Streets, J., {Complex manifolds with negative curvature operator},  Int.Math.Res.Notices,  rnz331.

\bibitem{Libermann1}
Libermann, P., Sur les connexions hermitiennes, C. R. Acad. Sci. Paris, \textbf{239} (1954), pp. 1579--1581.

\bibitem{Libermann2}
Libermann, P., Sur les structures presque-complexes et autres structures infinit\'esimales r\'eguli\`eres, Bull. Soc. Math. France, \textbf{83} (1955), pp. 195--224.

\bibitem{Libermann3}
Libermann, P., Classification and conformal properties of almost hermitian structures, Colloquia Mathematica Societatis J\'anos Bolyai, Differential Geometry \textbf{31}, Budapest (1979).

\bibitem{Lichnerowicz}
Lichnerowicz, A., Th\'eorie globale des connexions et des groupes d'holonomie, Edizioni Cremonese, Roma (1962).

\bibitem{LiuYangGeometry}
Liu, K., Yang, X., Geometry of Hermitian Manifolds, Internat. J. Math. 23 (2012), no. 6, 1250055 (40 pages)

\bibitem{LiuYangRicci}
Liu, K., Yang, X., Ricci curvatures on Hermitian manifolds. Trans. Amer. Math. Soc. 369 (2017), no. 7, 5157–5196.

\bibitem{LiuYangLCRF}
Liu, K., Yang, X.,  Minimal complex surfaces with Levi-Civita Ricci-flat metrics, Acta Math. Sin. (Engl. Ser.) 34 (2018), no. 8, 1195--1207. 

\bibitem{Lu}
Lu,  Y.-C., Holomorphic mappings of complex manifolds. J. Differential Geometry 2 (1968), 299--312

\bibitem{OrneaVerbitsky}
Ornea,  L., Verbitsky, M.,  Einstein-Weyl structures on complex manifolds and conformal version of Monge--Amp\`ere equation, Bull. Math. de la Soc. des Sci. Math. de Roumanie, Nouvelle S\'erie, vol. 51 (99), no. 4 (2008), pp. 339--353.

\bibitem{OrneaVerbitsky}
Ornea, L., Verbitsky, M., Principles of locally conformally K\"ahler geometry, book, arXiv:2208.07188


\bibitem{Royden}
Royden, H. L.,  The Ahlfors-Schwarz lemma in several complex variables, Comment. Math. Helvetici \textbf{55} (1980), p. 547--558.

\bibitem{Rubinstein}
Rubinstein, Y.,  Smooth and singular K\"ahler-Einstein metrics. Geometric and spectral analysis, 45–138, Contemp. Math., 630, Centre Rech. Math. Proc., Amer. Math. Soc., Providence, RI, 2014.


\bibitem{StreetsVII}
Streets, J., Pluriclosed flow and the geometrization of complex surfaces. Geometric analysis—in honor of Gang Tian's 60th birthday, 471–510, Progr. Math., 333, Birkh\"auser/Springer, Cham, 2020

\bibitem{StreetsTian}
Streets, J., Tian, G.,  Hermitian curvature flow. J. Eur. Math. Soc. (JEMS) 13 (2011), no. 3, 601--634.


\bibitem{StreetsTianPCF1}
Streets, J., Tian, G.,  A parabolic flow of pluriclosed metrics.  Int.  Math.  Res.  Not.  IMRN 2010, no. 16, 3101--3133

\bibitem{StreetsTianPCF2}
Streets, J., Tian, G., Regularity theory for pluriclosed flow, C. R. Acad. Sci. Paris, Ser. I 349 (2011), pp. 1--4.

\bibitem{Strominger1986}
Strominger,  A., Superstrings with torsion.  Nuclear Phys. B 274 (1986), no. 2, 253--284. 


\bibitem{TosattiNonKahlerCY}
Tosatti, V., Non-K\"ahler Calabi--Yau manifolds. Contemp. Math. \textbf{644} (2015), 261--277.


\bibitem{WangYang}
Wang, J., Yang, X., Curvatures of real connections on Hermitian manifolds, arXiv:1912.12024


\bibitem{YangZhengCurvature}
Yang, B., Zheng, F., {On curvature tensors of Hermitian manifolds}, Communications in Analysis and Geometry, vol. \textbf{26} (2018), no. 5, pp. 1195--1222



\bibitem{YangScalarCurvature}
Yang, X., Scalar curvature on compact complex manifolds. Trans. Amer. Math. Soc. 371 (2019), no. 3, 2073–2087.

\bibitem{YangKodairaDimension}
Yang, X., Scalar curvature, Kodaira dimension and $\hat{A}$-genus, Math. Z. 295 (2020), no. 1-2, 365–380

\bibitem{YangHSCYau}
Yang, X., RC-positivity, rational connectedness and Yau’s conjecture, Camb. J. Math. 6 (2018), no. 2, 183--212. 

\bibitem{YangSecondRicci}
Yang, X., Compact K\"ahler manifolds with quasi-positive second Chern-Ricci curvature, arXiv:2006.13884


\bibitem{YangZhengRBC}
Yang, X., Zheng, F., {On the real bisectional curvature for Hermitian manifolds}, Trans. Amer. Math. Soc. \textbf{371} (2019), no. 4, 2703--2718


\bibitem{YauSchwarz}
Yau, S. -T.,  A general Schwarz lemma for K\"ahler manifolds. Amer. J. Math. 100 (1978), no. 1, 197--203. MR0486659

\bibitem{ZhaoZheng}
Zhao, Q., Zheng, F., On Gauduchon K\"ahler-like manifolds, On Gauduchon Kähler-like manifolds. J. Geom. Anal. 32 (2022), no. 4, Paper No. 110, 27 pp.


\end{thebibliography}
\end{document}